\documentclass{amsart}

\usepackage[]{amsmath}
\usepackage{amsfonts}
\usepackage{amssymb}
\usepackage{mathrsfs}

\usepackage[all,knot]{xy}
\xyoption{arc}

\newtheorem{theorem}{Theorem}[section]
\newtheorem{lemma}[theorem]{Lemma}
\newtheorem{proposition}[theorem]{Proposition}
\newtheorem{corollary}[theorem]{Corollary}

\theoremstyle{definition}
\newtheorem{definition}[theorem]{Definition}
\newtheorem{example}[theorem]{Example}

\theoremstyle{remark}

\numberwithin{equation}{section}

\begin{document}

\title[A generalized Jordan canonical form theorem]
{A reduction theory for operators in type $\rm{I}_{n}$ von Neumann
algebras}


\author[R. Shi]{Rui Shi}


\address{Rui Shi\\
School of Mathematical Sciences, Dalian University of Technology,
Dalian, 116024, China }

\curraddr{
} \email{ruishi.math@gmail.com}








\begin{abstract}
In this paper, we study the structures of operators in a type
$\mathrm{I}^{}_{n}$ von Neumann algebra $\mathscr{A}$. Inspired by
the Jordan canonical form theorem, our main result is to describe
the structure of an operator $A$ in $\mathscr{A}$ with the property
that bounded maximal abelian sets of idempotents contained in the
relative commutant $\{A\}^{\prime}\cap\mathscr{A}$ are the same up
to similarity. We classify this class of operators whose relative
commutants possess the stated property by $K$-theory for Banach
algebras. We use techniques of von Neumann's reduction theory in our
proofs.
\end{abstract}

\maketitle

\section{Introduction} 

It is well-known that the Jordan canonical form theorem states that
each operator $A$ in $M^{}_{n}(\mathbb{C})$ is similar to a direct
sum of Jordan matrices and the direct sum is unique up to
similarity. An equivalent statement is that any two (bounded)
maximal abelian sets of idempotents $\mathscr{P}$ and $\mathscr{Q}$
in $\{A\}^{\prime}\cap M^{}_{n}(\mathbb{C})$ are similar to each
other in $\{A\}^{\prime}\cap M^{}_{n}(\mathbb{C})$.

There are two natural ways when we consider to generalize the Jordan
canonical form theorem. One way is to consider the generalization in
the type $\mathrm{I}^{}_{\infty}$ factor instead of the type
$\mathrm{I}^{}_{n}$ factor $M^{}_{n}(\mathbb{C})$. We started our
study in \cite{Shi_1} and carried on in \cite{Shi_2, Shi_3}. In this
`infinite' case, we found that there exists a normal operator $N$
such that there are two bounded maximal abelian sets of idempotents
$\mathscr{P}$ and $\mathscr{Q}$ in $\{N\}^{\prime}$ not similar to
each other in $\{N\}^{\prime}$. The multiplicity function for $N$
plays an important role, and it is required to be bounded if we want
bounded maximal abelian sets of idempotents in $\{N\}^{\prime}$ to
be the same up to similarity. Our results in \cite{Shi_2, Shi_3} are
proved based on the `bounded multiplicity' condition. On the other
hand, another way is to consider the generalization in type
$\mbox{I}^{}_{n}$ von Neumann algebras. It is also a natural
question to ask whether for every operator $A$ in a type
$\mbox{I}^{}_{n}$ von Neumann algebra $\mathscr{A}$, the property of
$\{A\}^{\prime}\cap\mathscr{A}$ holds that bounded maximal abelian
sets of idempotents in the relative commutant
$\{A\}^{\prime}\cap\mathscr{A}$ are the same up to similarity. If
the answer is negative, then we want to characterize the class of
operators in $\mathscr{A}$ whose relative commutants possess the
property. In this paper, we investigate the preceding questions and
the structures of operators in the class.

Throughout this article, we only discuss Hilbert spaces which are
{\textit{complex and separable}}. Denote by
$\mathscr{L}(\mathscr{H})$ the set of bounded linear operators on a
Hilbert space $\mathscr{H}$. Unless there is a danger of confusion,
we will assume from now on that $\mu$ is (the completion of) a
finite positive regular Borel measure supported on a compact subset
$\Lambda$ of $\mathbb {C}$. For the sake of simplicity, we consider
elements in $L^{\infty}(\mu)$ as multiplication operators on
$L^{2}(\mu)$ and matrices in $M^{}_{n}(L^{\infty}(\mu))$ as bounded
linear operators on $(L^{2}(\mu))^{(n)}$. In this sense every
operator $A$ in $M^{}_{n}(L^{\infty}(\mu))$ is in the form
$$A=\begin{pmatrix}
{f^{}_{11}}&\cdots&{f^{}_{1n}}\\
\vdots&\ddots&\vdots\\
{f^{}_{n1}}&\cdots&{f^{}_{nn}}\\
\end{pmatrix}^{}_{n\times n}
\begin{matrix}
L^{2}(\mu)\\
\vdots\\
L^{2}(\mu)\\
\end{matrix}, \eqno{(1.1)}$$
where the multiplication operator $M^{}_{f^{}_{ij}}$ is abbreviated
as $f^{}_{ij}$ in $L^{\infty}(\mu)$ and $i,j=1,\ldots,n$. An
\textit{idempotent} $P$ on $\mathscr{H}$ is an operator in
$\mathscr{L}(\mathscr{H})$ such that $P^{2}=P$. A
{\textit{projection}} $Q$ in $\mathscr{L}(\mathscr{H})$ is an
idempotent such that $Q=Q^{*}$. For an operator $A$ in
$M^{}_{n}(L^{\infty}(\mu))$, the relative commutant of $A$ with
respect to $M^{}_{n}(L^{\infty}(\mu))$ is denoted by
$\{A\}^{\prime}\cap M^{}_{n}(L^{\infty}(\mu))=\{B\in
M^{}_{n}(L^{\infty}(\mu)):AB=BA\}$.

For an operator $A$ in a type $\mbox{I}^{}_{n}$ von Neumann algebra
$M^{}_{n}(L^{\infty}(\mu))$, we need to introduce the following
definition.

\begin{definition}
Let $A$ be an operator in $M^{}_{n}(L^{\infty}(\mu))$. We say that
\textit{the strongly irreducible decomposition of $A$ is unique up
to similarity with respect to the relative commutant
$\{A\}^{\prime}\cap M^{}_{n}(L^{\infty}(\mu))$} if for every two
bounded maximal abelian sets of idempotents $\mathscr{P}$ and
$\mathscr{Q}$ in $\{A\}^{\prime}\cap M^{}_{n}(L^{\infty}(\mu))$,
there exists an invertible element $X$ in $\{A\}^{\prime}\cap
M^{}_{n}(L^{\infty}(\mu))$ such that
$X\mathscr{P}X^{-1}=\mathscr{Q}$. We abbreviate such a relative
commutant $\{A\}^{\prime}\cap M^{}_{n}(L^{\infty}(\mu))$ as a
relative commutant with Property `UDSR'.
\end{definition}

By Definition $1.1$, for every matrix $A$ in $M^{}_{n}(\mathbb{C})$,
the strongly irreducible decomposition of $A$ is unique up to
similarity with respect to $\{A\}^{\prime}\cap
M^{}_{n}(\mathbb{C})$. ( Our reason to focus on the relative
commutant $\{A\}^{\prime}\cap M^{}_{n}(\mathbb{C})$ is that
$M^{}_{n}(\mathbb{C})$ can be embedded into another $C^{*}$ algebra
$\mathfrak{A}$, and in this sense $\{A\}^{\prime}\cap\mathfrak{A}$
is not always equal to $\{A\}^{\prime}\cap M^{}_{n}(\mathbb{C})$.)

In the present paper, to generalize the Jordan canonical form
theorem for operators in $M^{}_{n}(L^{\infty}(\mu))$, the first
question we need to deal with is whether the relative commutant
$\{A\}^{\prime}\cap M^{}_{n}(L^{\infty}(\mu))$ contains a bounded
maximal abelian set of idempotents for every operator $A$ in
$M^{}_{n}(L^{\infty}(\mu))$. We find the answer to this question is
negative in Example $2.11$. Thus we characterize several necessary
and sufficient conditions for the relative commutant
$\{A\}^{\prime}\cap M^{}_{n}(L^{\infty}(\mu))$ containing a bounded
maximal abelian set of idempotents, and these conditions can also be
used to distinguish for which kinds of operators in
$M^{}_{n}(L^{\infty}(\mu))$, the strongly irreducible decompositions
of these operators are unique up to similarity with respect to the
relative commutants. One of our main theorems is as follows:

\begin{theorem}
Let $A$ be an operator in $M^{}_{n}(L^{\infty}(\mu))$. Then the
following statements are equivalent:
\begin{enumerate}
\item the relative commutant $\{A\}^{\prime}\cap
M^{}_{n}(L^{\infty}(\mu))$ contains a finite frame
$\{P^{}_{k}\}^{m}_{k=1}$;
\item the relative commutant $\{A\}^{\prime}\cap
M^{}_{n}(L^{\infty}(\mu))$ contains a bounded maximal abelian set of
idempotents;
\item the local structures of $A$ are stated as in Proposition $2.8$;
\item there exists an invertible element $X$
in $M^{}_{n^{}_{}}(L^{\infty}(\mu))$ and a unitary operator $U$ such
that $UXAX^{-1}U^{*}_{}$ is a direct integral of strongly
irreducible operators with respect to a diagonal algebra
$\mathscr{D}$ and $U^{*}_{}\mathscr{D}U\subseteq
M^{}_{n^{}_{}}(L^{\infty}(\mu))$.
\end{enumerate}
If one of the above conditions holds for $A$, then for every two
bounded maximal abelian sets of idempotents $\mathscr{P}$ and
$\mathscr{Q}$ in $\{A\}^{\prime}\cap M^{}_{n}(L^{\infty}(\mu))$,
there exists an invertible element $X$ in $\{A\}^{\prime}\cap
M^{}_{n}(L^{\infty}(\mu))$ such that
$X\mathscr{P}X^{-1}=\mathscr{Q}$.
\end{theorem}

Theorem $1.2$ is mainly a combination of Theorem $3.9$, Proposition
$2.5$, Proposition $2.8$ and Proposition $4.3$. For a building block
$J^{}_{n^{}_{k}}$ in $M^{}_{n^{}_{k}}(L^{\infty}(\mu))$ stated in
the `Local Structure Proposition' (Proposition $2.8$), we compute
the $K^{}_{0}$ group of $\{J^{}_{n^{}_{k}}\}^{\prime}\cap
M^{}_{n^{}_{k}}(L^{\infty}(\mu))$ in Lemma $4.4$. On the other hand,
by Proposition $2.8$ we apply the $K^{}_{0}$ groups of the relative
commutants to classify operators $A$ in $M^{}_{n}(L^{\infty}(\mu))$
with the property that each relative commutant $\{A\}^{\prime}\cap
M^{}_{n}(L^{\infty}(\mu))$ contains a finite frame.

The present paper is organized as follows. In section $2$, we prove
several preliminary lemmas and introduce the concept for a frame to
be finite in a relative commutant. In Proposition $2.5$, we prove
that the statements $(1)$ and $(2)$ in Theorem $1.2$ are equivalent.
In Proposition $2.8$, we characterize the local structures of $A$
with respect to the center of $M^{}_{n}(L^{\infty}(\mu))$, for an
operator $A$ in $M^{}_{n}(L^{\infty}(\mu))$ with $\{A\}^{\prime}\cap
M^{}_{n}(L^{\infty}(\mu))$ containing a finite frame. Then we
present an example in which an operator in
$M^{}_{2}(L^{\infty}(\mu))$ is constructed such that the relative
commutant contains no finite frames. By this example and the proof
of Proposition $2.8$, we can construct many examples of this type.
Section $3$ is mainly devoted to the proof Theorem $3.9$. As a
corollary, we prove that every normal operator in
$M^{}_{n}(L^{\infty}(\mu))$ possesses the properties mentioned in
Theorem $1.2$. In section $4$, with the aid of Section $2$ and
Section $3$, we show the connection between the direct integrals of
strongly irreducible operators and the operators in
$M^{}_{n}(L^{\infty}(\mu))$ with the property mentioned in $(2)$ of
Theorem $1.2$. Then we discuss the `local' $K$-theory of the
relative commutant of $A$ in $M^{}_{n}(L^{\infty}(\mu))$ with
respect to the center of $M^{}_{n}(L^{\infty}(\mu))$ and Proposition
$2.8$. By virtue of the discussion, we prove that the `local'
$K$-theory of the relative commutant of $A$ in
$M^{}_{n}(L^{\infty}(\mu))$ can be used as a complete similarity
invariant to classify operators with the properties mentioned in
Theorem $1.2$.

\section{The local structures of $A$ in
$M^{}_{n}(L^{\infty}(\mu))$ for $\{A\}^{\prime}\cap
M^{}_{n}(L^{\infty}(\mu))$ containing a finite frame}

The following two lemmas are devoted to proving Lemma $2.3$. In the
proof of Lemma $2.3$, we introduce an algorithm to find an
invertible operator $X$ in $M^{}_{n}(L^{\infty}(\mu))$ such that
$XPX^{-1}$ is diagonal, for a given idempotent $P$ in
$M^{}_{n}(L^{\infty}(\mu))$. This algorithm is also used in the
proof of the main theorem.

\begin{lemma}
Let $P$ be a non-zero idempotent in $M^{}_{n}(\mathbb {C})$ of the
form
$$P=\begin{pmatrix}
\alpha^{}_{11}&\alpha^{}_{12}&\cdots&\alpha^{}_{1n}\\
\alpha^{}_{21}&\alpha^{}_{22}&\cdots&\alpha^{}_{2n}\\
\vdots&\vdots&\ddots&\vdots\\
\alpha^{}_{n1}&\alpha^{}_{n2}&\cdots&\alpha^{}_{nn}\\
\end{pmatrix}^{}_{n\times n}, \eqno{(2.1)}$$
where $\mathrm{rank} (P)=r>0$ and $\alpha^{}_{11}\neq 0$. Let $X$ be
a lower triangular matrix of the form
$$X=\begin{pmatrix}
1&0&\cdots&0\\
-\dfrac{\alpha^{}_{21}}{\alpha^{}_{11}}&1&\cdots&0\\
\vdots&\vdots&\ddots&\vdots\\
-\dfrac{\alpha^{}_{n1}}{\alpha^{}_{11}}&0&\cdots&1\\
\end{pmatrix}^{}_{n\times n}. \eqno{(2.2)}$$
Then $X$ is invertible in $M^{}_{n}(\mathbb {C})$ such that the
$(1,1)$ entry of $XPX^{-1}$ is $1$ and the $(i,1)$ entry of
$XPX^{-1}$ is $0$ for $i=2,\ldots,n$.
\end{lemma}

\begin{proof}
If the rank of $P$ is $n$, then we obtain $P=I^{}_{n}$, where we
denote by $I^{}_{n}$ the unit of $M^{}_{n}(\mathbb {C})$. For
$0<\mathrm{rank} (P)<n$, by a computation, we obtain that $X^{-1}$
is of the form
$$X^{-1}_{}=\begin{pmatrix}
1&0&\cdots&0\\
\dfrac{\alpha^{}_{21}}{\alpha^{}_{11}}&1&\cdots&0\\
\vdots&\vdots&\ddots&\vdots\\
\dfrac{\alpha^{}_{n1}}{\alpha^{}_{11}}&0&\cdots&1\\
\end{pmatrix}^{}_{n\times n}. \eqno{(2.3)}$$

Notice that $\mathrm{rank}(P)=\mathrm{rank} (XP)$ and the $(i,1)$
entry of $XP$ is $0$ for $i=2,\ldots,n$. Without loss of generality,
we assume that $XP$ is of the form
$$XP=\begin{pmatrix}
\beta^{}_{11}&\beta^{}_{12}&\cdots&\beta^{}_{1n}\\
0&\beta^{}_{22}&\cdots&\beta^{}_{2n}\\
\vdots&\vdots&\ddots&\vdots\\
0&\beta^{}_{n2}&\cdots&\beta^{}_{nn}\\
\end{pmatrix}^{}_{n\times n}, \eqno{(2.4)}$$
and the first $r$ rows $\{\beta^{}_{i}\}^{r}_{i=1}$ of $XP$ are
linear independent, where we denote by $\beta^{}_{i}$ the $i$-th row
of $XP$. Note that $\alpha^{}_{1i}=\beta^{}_{1i}$ for
$i=1,\ldots,n$.

We assert that every element of $\{\beta^{}_{i}\}^{n}_{i=r+1}$ is a
linear combination of $\beta^{}_{2},\ldots,\beta^{}_{r}$.

By the foregoing assumption, every element of
$\{\beta^{}_{i}\}^{n}_{i=r+1}$ is a linear combination of
$\beta^{}_{1},\ldots,\beta^{}_{r}$. Assume that
$$\beta^{}_{r+s}=\lambda^{}_{1}\beta^{}_{1}+\cdots+\lambda^{}_{r}\beta^{}_{r}
\eqno{(2.5)}$$ where $\lambda^{}_{i}\in\mathbb {C}$ for
$i=1,\ldots,r$, $s=1,\ldots,n-r$. If $\lambda^{}_{1}\neq 0$, then
the $(r+s,1)$ entry of $XP$ will not be $0$. This contradicts
$(2.4)$. Thus we obtain the assertion.

By the preceding assertion, there exists a lower triangular
invertible matrix $Y$ of the form
$$Y=\begin{pmatrix}
1&0&\cdots&0&0&\cdots&0\\
0&1&\cdots&0&0&\cdots&0\\
\vdots&\vdots&\ddots&\vdots&\vdots&\ddots&\vdots\\
0&0&\cdots&1&0&\cdots&0\\
0&\lambda^{}_{r+1,2}&\cdots&\lambda^{}_{r+1,r}&1&\cdots&0\\
\vdots&\vdots&\ddots&\vdots&\vdots&\ddots&\vdots\\
0&\lambda^{}_{n,2}&\cdots&\lambda^{}_{n,r}&0&\cdots&1\\
\end{pmatrix}^{}_{n\times n} \eqno{(2.6)}$$ such that the $(r+s)$-th
row of $YXP$ is $\mathbf{0}$ for $s=1,\ldots,n-r$.

Note that $YXPX^{-1}Y^{-1}$ is an idempotent of the form
$$YXPX^{-1}Y^{-1}=\begin{pmatrix}
P^{}_{11}&R\\
\mathbf{0}&\mathbf{0}\\
\end{pmatrix}
\begin{array}{l}
\mathbb{C}^{(r)}\\
\mathbb{C}^{(n-r)}\\
\end{array}, \eqno{(2.7)}$$ where $P^{}_{11}$ is in
$M^{}_{r}(\mathbb{C})$. Thus $P^{}_{11}$ is an idempotent. Note that
the standard trace of an idempotent $Q$ (denoted by
$\mathrm{Tr}(Q)$) in $M^{}_{n}(\mathbb {C})$ is equal to the rank of
$Q$. Therefore, we obtain that
$$\begin{array}{rcl}
r=\mathrm{rank}(P)&=&\mathrm{rank}(YXPX^{-1}Y^{-1})\\
&=&\mathrm{Tr}(YXPX^{-1}Y^{-1})\\
&=&\mathrm{Tr}(P^{}_{11})\\
&=&\mathrm{rank}(P^{}_{11}).\\
\end{array} \eqno{(2.8)}$$
Thus $P^{}_{11}=I^{}_{r}$ is the unit of $M^{}_{r}(\mathbb {C})$. By
the construction of $Y$, the $(i,1)$ entries of $XPX^{-1}$,
$YXPX^{-1}$ and $YXPX^{-1}Y^{-1}$ are the same for $i=1,\ldots,r$.
Since $P^{}_{11}=I^{}_{r}$, we obtain that the $(1,1)$ entry of
$XPX^{-1}$ is $1$ and the $(i,1)$ entry of $XPX^{-1}$ is $0$ for
$i=2,\ldots,r$. The equality $(YXP)X^{-1}=Y(XPX^{-1})$ yields that
the $(i,1)$ entry of $YXPX^{-1}$ is $0$ for $i=r+1,\ldots,n$. By the
construction of $Y$ and the fact that the $(i,1)$ entry of
$XPX^{-1}$ is $0$ for $i=2,\ldots,r$, we obtain that the $(i,1)$
entry of $XPX^{-1}$ is $0$ for $i=r+1,\ldots,n$. The proof is
finished.
\end{proof}

It is well-known that in $M^{}_{n}(\mathbb {C})$ every idempotent is
similar to a diagonal projection. The reason that we restate this in
the following lemma is to develop an algorithm which leads to a
solution of a related problem raised in $M^{}_{n}(L^{\infty}(\mu))$.

\begin{lemma}
If $P$ is an idempotent in $M^{}_{n}(\mathbb {C})$, then there
exists an invertible matrix $X$ in $M^{}_{n}(\mathbb {C})$ such that
$XPX^{-1}$ is diagonal in $M^{}_{n}(\mathbb {C})$, where $X$ is a
composition of finitely many invertible matrices as in $(2.2)$ and
row-switching unitary matrices and an invertible block matrix of the
form
$$\begin{pmatrix}
I^{}_{r}&R\\
\mathbf{0}&I^{}_{n-r}\\
\end{pmatrix}
\begin{array}{l}
\mathbb{C}^{(r)}\\
\mathbb{C}^{(n-r)}\\
\end{array}. \eqno{(2.9)}$$
\end{lemma}

\begin{proof}
If the idempotent $P$ is trivial in $M^{}_{n}(\mathbb {C})$, then
the invertible matrix $X$ can be chosen to be the unit. If $P$ is
nontrivial, then we assume $0<\mathrm{rank}(P)=r<n$. This yields
$0<\mathrm{Tr}(P)=\mathrm{rank}(P)=r$. Thus there exists a nonzero
entry denoted by $(i,i)$ in the main diagonal of $P$. Let $U^{}_{1}$
be the elementary matrix switching all matrix elements on row $1$
with their counterparts on row $i$. Then the $(1,1)$ entry of
$U^{}_{1}PU^{*}_{1}$ equals the $(i,i)$ entry of $P$. With respect
to $U^{}_{1}PU^{*}_{1}$, we construct an invertible operator
$X^{}_{1}$ as in $(2.2)$. Then
$P^{}_{1}(=X^{}_{1}U^{}_{1}P(X^{}_{1}U^{}_{1})^{-1})$ can be
expressed in the form
$$P^{}_{1}=\begin{pmatrix}
1&\alpha^{}_{12}&\cdots&\alpha^{}_{1n}\\
0&\alpha^{}_{22}&\cdots&\alpha^{}_{2n}\\
\vdots&\vdots&\ddots&\vdots\\
0&\alpha^{}_{n2}&\cdots&\alpha^{}_{nn}\\
\end{pmatrix}^{}_{n\times n}. \eqno{(2.10)}$$ Let $P^{}_{22}$ be the
matrix of the form
$$P^{}_{22}=\begin{pmatrix}
\alpha^{}_{22}&\cdots&\alpha^{}_{2n}\\
\vdots&\ddots&\vdots\\
\alpha^{}_{n2}&\cdots&\alpha^{}_{nn}\\
\end{pmatrix}^{}_{(n-1)\times(n-1)}. \eqno{(2.11)}$$ Then $P^{}_{22}$
is an idempotent in $M^{}_{n-1}(\mathbb {C})$ and
$\mathrm{Tr}(P^{}_{22})=r-1$. If $r=1$, then $P^{}_{22}=0$. If
$r>1$, then there exists a nonzero entry $\alpha^{}_{jj}$ in the
main diagonal of $P^{}_{22}$. Let $U^{}_{2}$ be the elementary
matrix in $M^{}_{n}(\mathbb {C})$ switching all matrix elements on
row $2$ with their counterparts on row $j$. Then the $(2,2)$ entry
of $U^{}_{2}X^{}_{1}U^{}_{1}P(U^{}_{2}X^{}_{1}U^{}_{1})^{-1}$ equals
the $(j,j)$ entry of $X^{}_{1}U^{}_{1}P(X^{}_{1}U^{}_{1})^{-1}$.
With respect to $P{}_{22}$, we construct an invertible operator
$\hat{X}^{}_{2}$ as in $(2.2)$. Then
$\hat{X}^{}_{2}P{}_{22}\hat{X}^{-1}_{2}$ can be expressed in the
form
$$\hat{X}^{}_{2}P{}_{22}\hat{X}^{-1}_{2}=\begin{pmatrix}
1&\beta^{}_{23}&\cdots&\beta^{}_{2n}\\
0&\beta^{}_{33}&\cdots&\beta^{}_{3n}\\
\vdots&\vdots&\ddots&\vdots\\
0&\beta^{}_{n3}&\cdots&\beta^{}_{nn}\\
\end{pmatrix}^{}_{(n-1)\times(n-1)}. \eqno{(2.12)}$$ Let $X^{}_{2}$
be of the form
$$X^{}_{2}=\begin{pmatrix}
1&0\\
0&\hat{X}^{}_{2}\\
\end{pmatrix}
\begin{matrix}
\mathbb{C}\\
\mathbb{C}^{(n-1)}_{}\\
\end{matrix}. \eqno{(2.13)}$$ Then
$P^{}_{2}(=X^{}_{2}U^{}_{2}X^{}_{1}U^{}_{1}P(X^{}_{2}U^{}_{2}X^{}_{1}U^{}_{1})^{-1})$
can be expressed in the form
$$P^{}_{2}=\begin{pmatrix}
1&0&\beta^{}_{13}&\cdots&\beta^{}_{1n}\\
0&1&\beta^{}_{23}&\cdots&\beta^{}_{2n}\\
0&0&\beta^{}_{33}&\cdots&\beta^{}_{3n}\\
\vdots&\vdots&\vdots&\ddots&\vdots\\
0&0&\beta^{}_{n3}&\cdots&\beta^{}_{nn}\\
\end{pmatrix}^{}_{(n\times n)}. \eqno{(2.14)}$$
After $r$ steps, we obtain $P^{}_{r}$ of the form
$$P^{}_{r}=\begin{pmatrix}
I^{}_{r}&R\\
\mathbf{0}&\mathbf{0}^{}_{n-r}\\
\end{pmatrix}
\begin{array}{l}
\mathbb{C}^{(r)}\\
\mathbb{C}^{(n-r)}\\
\end{array}. \eqno{(2.15)}$$ Let $X^{}_{r+1}$ be of the form
$$X^{}_{r+1}=\begin{pmatrix}
I^{}_{r}&R\\
\mathbf{0}&I^{}_{n-r}\\
\end{pmatrix}
\begin{array}{l}
\mathbb{C}^{(r)}\\
\mathbb{C}^{(n-r)}\\
\end{array}. \eqno{(2.16)}$$ Then $X^{}_{r+1}$ is invertible and
$X^{-1}_{r+1}$ is of the form
$$X^{-1}_{r+1}=\begin{pmatrix}
I^{}_{r}&-R\\
\mathbf{0}&I^{}_{n-r}\\
\end{pmatrix}
\begin{array}{l}
\mathbb{C}^{(r)}\\
\mathbb{C}^{(n-r)}\\
\end{array}. \eqno{(2.17)}$$ Thus $X^{}_{r+1}P^{}_{r}X^{-1}_{r+1}$
is diagonal in $M^{}_{n}(\mathbb {C})$. And we obtain the invertible
matrix $X=X^{}_{r+1}X^{}_{r}U^{}_{r}\cdots X^{}_{1}U^{}_{1}$ such
that $XPX^{-1}$ is diagonal in $M^{}_{n}(\mathbb {C})$.
\end{proof}

With the algorithm in the foregoing two lemmas, we prove the
following lemma.

\begin{lemma}
If $P$ is an idempotent in $M^{}_{n}(L^{\infty}(\mu))$, then there
exists an invertible operator $X$ in $M^{}_{n}(L^{\infty}(\mu))$
such that $XPX^{-1}$ is diagonal.
\end{lemma}

\begin{proof}
For the sake of simplicity, we write $P$ of the form
$$P=\begin{pmatrix}
f^{}_{11}&\cdots&f^{}_{1n}\\
\vdots&\ddots&\vdots\\
f^{}_{n1}&\cdots&f^{}_{nn}\\
\end{pmatrix}^{}_{n\times n}
\begin{matrix}
L^{2}(\mu)\\
\vdots\\
L^{2}(\mu)\\
\end{matrix}, \eqno{(2.18)}$$
where every $f^{}_{ij}$ is in $L^{\infty}(\mu)$ for
$i,j=1,\ldots,n$. We use the relaxed convention of treating Borel
representatives $f^{}_{ij}$ as elements in $L^{\infty}(\mu)$ for
$i,j=1,\ldots,n$ such that $P(\lambda)$ is an idempotent in
$M^{}_{n}(\mathbb {C})$ for every $\lambda$ in $\Lambda$. Thus we
obtain that
$$f=\sum^{n}_{i=1}f^{}_{ii} \eqno{(2.19)}$$ is an integer-valued
Borel simple function. With respect to the spectrum of $f$, we
obtain a Borel partition $\{\Lambda^{}_{k}\}^{n}_{k=0}$ of $\Lambda$
such that the standard trace of $P(\lambda)$ is $k$ for every
$\lambda$ in $\Lambda^{}_{k}$ and $k=0,\ldots,n$. Notice that
$\Lambda^{}_{k}$ may be of $\mu$-measure zero for some $k$s. Without
loss of generality, we assume $\Lambda=\Lambda^{}_{r}$ for $1\leq
r\leq n-1$.

Before we apply Lemma $2.1$ and Lemma $2.2$ to construct the
invertible operator, we need to modify the form in $(2.18)$.

We observe that there exists a Borel function
$f^{}_{i^{}_{1}i^{}_{1}}$ in the main diagonal of $P$ such that the
Borel subset $\Lambda^{}_{r1}$ of $\Lambda^{}_{r}$ of the form
$$\Lambda^{}_{r1}\triangleq
\{\lambda\in\Lambda^{}_{r}:|f^{}_{i^{}_{1}i^{}_{1}}(\lambda)|\geq
\frac{r}{n}\} \eqno{(2.20)}$$ is not of $\mu$-measure zero. In a
similar way, there exists a Borel function $f^{}_{i^{}_{2}i^{}_{2}}$
in the main diagonal of $P$ such that the Borel subset
$\Lambda^{}_{r2}$ of $\Lambda^{}_{r}\backslash\Lambda^{}_{r1}$ of
the form
$$\Lambda^{}_{r2}\triangleq
\{\lambda\in\Lambda^{}_{r}\backslash\Lambda^{}_{r1}:
|f^{}_{i^{}_{2}i^{}_{2}}(\lambda)|\geq \frac{r}{n}\} \eqno{(2.21)}$$
is not of $\mu$-measure zero. In this way, we obtain a Borel
partition $\{\Lambda^{}_{rj}\}^{k}_{j=1}$ of $\Lambda^{}_{r}$ with
$k\leq n$. There exists a unitary operator $U^{}_{1}$ in
$M^{}_{n}(L^{\infty}(\mu))$ such that $U^{}_{1}(\lambda)$ is an
elementary matrix in $M^{}_{n}(\mathbb{C})$ switching all matrix
entries in row $i^{}_{j}$ with their counterparts in row $1$ for
every $\lambda$ in $\Lambda^{}_{rj}$ and $j=1,\ldots,k$. Thus the
absolute value of the $(1,1)$ entry of $U^{}_{1}PU^{*}_{1}(\lambda)$
for every $\lambda$ in $\Lambda^{}_{r}$ is not less than $rn^{-1}$.
We write the operator $U^{}_{1}PU^{*}_{1}$ of the form
$$U^{}_{1}PU^{*}_{1}=\begin{pmatrix}
h^{}_{11}&\cdots&h^{}_{1n}\\
\vdots&\ddots&\vdots\\
h^{}_{n1}&\cdots&h^{}_{nn}\\
\end{pmatrix}^{}_{n\times n}
\begin{matrix}
L^{2}(\mu)\\
\vdots\\
L^{2}(\mu)\\
\end{matrix}. \eqno{(2.22)}$$ Thus we obtain
$|h^{}_{11}(\lambda)|\geq\frac{r}{n}$ for every $\lambda$ in
$\Lambda^{}_{r}$. Let $X^{}_{1}$ be constructed as in $(2.2)$
$$X^{}_{1}=\begin{pmatrix}
\mathbf{1}&0&\cdots&0\\
-{h^{}_{21}}\slash{h^{}_{11}}&\mathbf{1}&\cdots&0\\
\vdots&\vdots&\ddots&\vdots\\
-{h^{}_{n1}}\slash{h^{}_{11}}&0&\cdots&\mathbf{1}\\
\end{pmatrix}^{}_{n\times n}
\begin{matrix}
L^{2}(\mu)\\
L^{2}(\mu)\\
\vdots\\
L^{2}(\mu)\\
\end{matrix}. \eqno{(2.23)}$$
Write $P^{}_{1}=X^{}_{1}U^{}_{1}PU^{-1}_{1}X^{-1}_{1}$. By Lemma
$2.1$, $P^{}_{1}$ can be expressed in the form
$$P^{}_{1}=\begin{pmatrix}
\mathbf{1}&\varphi^{}_{12}&\cdots&\varphi^{}_{1n}\\
0&\varphi^{}_{22}&\cdots&\varphi^{}_{2n}\\
\vdots&\vdots&\ddots&\vdots\\
0&\varphi^{}_{n2}&\cdots&\varphi^{}_{nn}\\
\end{pmatrix}^{}_{n\times n}
\begin{matrix}
L^{2}(\mu)\\
L^{2}(\mu)\\
\vdots\\
L^{2}(\mu)\\
\end{matrix}. \eqno{(2.24)}$$ Combining the preceding construction of
$U^{}_{1}$ and the algorithm developed in the proof of Lemma $2.2$,
we can construct an invertible operator $X$ in
$M^{}_{n}(L^{\infty}(\mu))$ dependent on $P$ such that $XPX^{-1}$ is
diagonal.
\end{proof}

By applying the preceding lemma, we define a $\mu$-measurable
function $\mbox{Tr}(P)$ of the form
$$\mbox{Tr}(P)(\lambda)\triangleq\mbox{Tr}(P(\lambda))=\mbox{rank}(P(\lambda))
\eqno{(2.25)}$$ for every idempotent $P$ in
$M^{}_{n}(L^{\infty}(\mu))$ and almost every $\lambda$ in $\Lambda$,
where $\mbox{Tr}$ is the standard trace on $P(\lambda)$. Denote by
$\mathscr{E}^{}_{n}$ the set of central projections in
$M^{}_{n}(L^{\infty}(\mu))$. For every central projection $E$ in
$\mathscr{E}^{}_{n}$ there exists a Borel subset $\Lambda^{}_{E}$ of
$\Lambda$ such that
$$\mbox{Tr}(E^{}_{})(\lambda)=\begin{cases}
n,&\mbox{if }\lambda\in \Lambda^{}_{E},\\
0,&\mbox{if }\lambda\in\Lambda\backslash\Lambda^{}_{E}.
\end{cases} \eqno{(2.26)}$$

On generalizing the Jordan canonical form theorem, we need to
introduce the concept `finite frame' in the relative commutant of an
operator $A$ in $M^{}_{n}(L^{\infty}(\mu))$.

\begin{definition}
Let $A$ be an operator in $M^{}_{n}(L^{\infty}(\mu))$. A finite
subset $\{P^{}_{k}\}^{m}_{k=1}$ of idempotents in
$\{A\}^{\prime}\cap M^{}_{n}(L^{\infty}(\mu))$ is said to be a
\textit{finite frame} of $\{A\}^{\prime}\cap
M^{}_{n}(L^{\infty}(\mu))$ if the following conditions are
satisfied:
\begin{enumerate}
\item $P^{}_{i}P^{}_{j}=P^{}_{j}P^{}_{i}=0$ for $i\neq j$ and
$i,j=1,\dots,k$;
\item the idempotent $\sum^{m}_{k=1} P^{}_{k}$ equals to the identity of
$M^{}_{n}(L^{\infty}(\mu))$;
\item $P^{}_{k}$ is minimal in the set
$\{P\in\mathscr{P}^{}_{k}:P(\lambda)\neq 0 \mbox{ a.e. } [\mu]
\mbox{ on } \Lambda^{}_{k}\}$, where we write
$\mathscr{P}^{}_{k}=\{P\in\{A\}^{\prime}\cap
M^{}_{n}(L^{\infty}(\mu)):P^{2}=P, PP^{}_{k}=P^{}_{k}P\}$ and
$\Lambda^{}_{k}=\{\lambda\in\Lambda:P^{}_{k}(\lambda)\neq 0\}$ is
the support of $\mbox{Tr}(P^{}_{k})$;
\item the first three items also hold for every
$\{P^{}_{k}E^{}_{}\}^{m}_{k=1}$ restricted on $\mathrm{ran}E^{}_{}$,
where $E^{}_{}$ is a central projection and $\mathrm{Tr}(E^{}_{})$
is supported on a Borel subset $\Lambda^{}_{E}$ of $\Lambda$.
\end{enumerate}
Furthermore, a finite frame $\{P^{}_{k}\}^{m}_{k=1}$ in
$\{A\}^{\prime}\cap M^{}_{n}(L^{\infty}(\mu))$ is said to be
\textit{self-adjoint} if $P^{}_{k}$ is self-adjoint for
$k=1,\ldots,m$.
\end{definition}

This concept is inspired by the `cross section' in fibre bundles.
The difference is that every element in a finite frame here is Borel
measurable. In the rest of this section, we establish a relation
between the local structures of $A$ in $M^{}_{n}(L^{\infty}(\mu))$
and a finite frame in $\{A\}^{\prime}\cap
M^{}_{n}(L^{\infty}(\mu))$. First, we give a necessary and
sufficient condition for $\{A\}^{\prime}\cap
M^{}_{n}(L^{\infty}(\mu))$ containing a finite frame.

\begin{proposition}
Let $A$ be an operator in $M^{}_{n}(L^{\infty}(\mu))$. Then
$\{A\}^{\prime}\cap M^{}_{n}(L^{\infty}(\mu))$ contains a finite
frame if and only if $\{A\}^{\prime}\cap M^{}_{n}(L^{\infty}(\mu))$
contains a bounded maximal abelian set of idempotents $\mathscr{P}$.
\end{proposition}

\begin{proof}
Note that if $\{P^{}_{k}\}^{m}_{k=1}$ is a finite frame of
$\{A\}^{\prime}\cap M^{}_{n}(L^{\infty}(\mu))$, then
$\{P^{}_{k}\}^{m}_{k=1}$ and $\mathscr{E}^{}_{n}$ generate a bounded
maximal abelian set of idempotents $\mathscr{P}$ in
$\{A\}^{\prime}\cap M^{}_{n}(L^{\infty}(\mu))$. Actually, every
idempotent $P$ in $\mathscr{P}$ is a finite combination of
idempotents $\{P^{}_{k}\}^{m}_{k=1}$ cut by some central
projections.

On the other hand, if $\mathscr{P}$ is a bounded maximal abelian set
of idempotents in $\{A\}^{\prime}\cap M^{}_{n}(L^{\infty}(\mu))$,
then we write
$\mathscr{P}^{}_{0}=\{P\in\mathscr{P}:\mbox{Tr}(P)\mbox{ is
supported on }\Lambda\}$. The set $\mathscr{P}^{}_{0}$ is not empty,
since $\mathscr{P}^{}_{0}$ contains the identity of
$M^{}_{n}(L^{\infty}(\mu))$. Note that every $\mbox{Tr}(P)$ is a
$\mu$-measurable integer-valued simple function for $P$ in
$\mathscr{P}^{}_{}$. Thus there exists a maximal totally-ordered
subset of $\mathscr{P}^{}_{0}$. Since $\mathscr{P}^{}_{}$ is bounded
and closed in the weak-operator topology, there exists a minimal
idempotent $P^{}_{1}$ in $\mathscr{P}^{}_{0}$, which means there is
no proper sub-idempotents of $P^{}_{1}$ in $\mathscr{P}^{}_{0}$. If
$P^{}_{1}=I$, then $\{P^{}_{1}\}$ is a finite frame of
$\mathscr{P}$. Otherwise, denote by $\Lambda^{}_{1}$ the support of
$I-P^{}_{1}$ and write
$\mathscr{P}^{}_{1}=\{P\in(I-P^{}_{1})\mathscr{P}^{}_{}:P\mbox{ is
supported on }\Lambda^{}_{1}\}$. Thus there exists a bounded maximal
totally-ordered subset of $\mathscr{P}^{}_{1}$. By reduction, there
exists a minimal idempotent $P^{}_{2}$ in $\mathscr{P}^{}_{1}$. If
$P^{}_{2}\neq I-P^{}_{1}$, then we iterate the preceding procedure.
After finite steps the procedure stops and we obtain a finite subset
$\{P^{}_{k}\}^{m}_{k=1}$ in $\mathscr{P}^{}_{}$. It's obvious that
the set $\{P^{}_{k}E\}^{m}_{k=1}$ satisfies the first three items
restricted on $\mathrm{ran}E$ in Definition $2.4$ for every central
projection $E$ in $M^{}_{n}(L^{\infty}(\mu))$. Therefore,
$\{P^{}_{k}\}^{m}_{k=1}$ is a required finite frame in
$\mathscr{P}^{}_{}$.
\end{proof}

In the following lemma, for an operator $A$ in
$M^{}_{n}(L^{\infty}(\mu))$ such that the relative commutant
$\{A\}^{\prime}\cap M^{}_{n}(L^{\infty}(\mu))$ contains a finite
frame $\{P^{}_{k}\}^{m}_{k=1}$, we modify the finite frame to be
self-adjoint.

\begin{lemma}
For an operator $A$ in $M^{}_{n}(L^{\infty}(\mu))$ such that
$\{A\}^{\prime}\cap M^{}_{n}(L^{\infty}(\mu))$ contains a finite
frame $\{P^{}_{k}\}^{m}_{k=1}$, there exists an invertible operator
$X$ in $M^{}_{n}(L^{\infty}(\mu))$ such that every $XP^{}_{k}X^{-1}$
is diagonal for $k=1,\ldots,m$.
\end{lemma}

\begin{proof}
Assume that $A$ is an operator in $M^{}_{n}(L^{\infty}(\mu))$ such
that $\{A\}^{\prime}\cap M^{}_{n}(L^{\infty}(\mu))$ contains a
finite frame $\{P^{}_{k}\}^{m}_{k=1}$. There exists a finite Borel
partition $\{\Lambda^{}_{s}\}^{r}_{s=1}$ of $\Lambda$ such that for
every central projection $E^{}_{\Lambda^{}_{s}}$ as in $(2.26)$ and
$k=1,\ldots,m$, the function
$\mbox{Tr}(P^{}_{k}E^{}_{\Lambda^{}_{s}})$ takes a constant
a.e.~$[\mu]$ on $\Lambda^{}_{s}$. Therefore without loss of
generality, we assume that $\mbox{Tr}(P^{}_{k})(\lambda)=r^{}_{k}$
(a constant function) a.e.~$[\mu]$ on $\Lambda^{}_{}$ for
$k=1,\ldots,m$. By Lemma $2.3$, there exists an invertible operator
$X^{}_{1}$ in $M^{}_{n}(L^{\infty}(\mu))$ such that
$X^{}_{1}P^{}_{1}X^{-1}_{1}$ is diagonal and
$X^{}_{1}P^{}_{k}X^{-1}_{1}$ is of the form
$$X^{}_{1}P^{}_{k}X^{-1}_{1}=\begin{pmatrix}
0&0\\
0&Q^{}_{k}\\
\end{pmatrix}
\begin{array}{l}
(L^{2}(\mu))^{(r^{}_{1})}\\
(L^{2}(\mu))^{(n-r^{}_{1})}\\
\end{array}, \eqno{(2.27)}$$
where $Q^{}_{k}$ is an idempotent in
$M^{}_{n-r^{}_{1}}(L^{\infty}(\mu))$ for $k=2,\ldots,m$. Again by
Lemma $2.3$, there exists an invertible operator $Y^{}_{2}$ in
$M^{}_{n-r^{}_{1}}(L^{\infty}(\mu))$ such that
$Y^{}_{2}Q^{}_{2}Y^{-1}_{2}$ is diagonal and
$Y^{}_{2}Q^{}_{k}Y^{-1}_{2}$ is of the form
$$Y^{}_{2}Q^{}_{k}Y^{-1}_{2}=\begin{pmatrix}
0&0\\
0&R^{}_{k}\\
\end{pmatrix}
\begin{array}{l}
(L^{2}(\mu))^{(r^{}_{2})}\\
(L^{2}(\mu))^{(n-r^{}_{1}-r^{}_{2})}\\
\end{array}, \eqno{(2.28)}$$
where $R^{}_{k}$ is an idempotent in
$M^{}_{n-r^{}_{1}-r^{}_{2}}(L^{\infty}(\mu))$ for $k=3,\dots,m$. Let
$X^{}_{2}$ be of the form
$$X^{}_{2}=\begin{pmatrix}
I&0\\
0&Y^{}_{2}\\
\end{pmatrix}
\begin{array}{l}
(L^{2}(\mu))^{(r^{}_{1})}\\
(L^{2}(\mu))^{(n-r^{}_{1})}\\
\end{array}. \eqno{(2.29)}$$ Then $X^{}_{2}$ is invertible in
$M^{}_{n}(L^{\infty}(\mu))$. Iterating the preceding procedure, we
obtain $m-1$ invertible operators $\{X^{}_{k}\}^{m-1}_{k=1}$ in
$M^{}_{n}(L^{\infty}(\mu))$. The invertible operator
$X=X^{}_{m-1}\cdots X^{}_{1}$ is as required in this lemma.
\end{proof}

By Lemma $2.6$, we investigate the local structures of $A$ in
$M^{}_{n}(L^{\infty}(\mu))$ if the relative commutant
$\{A\}^{\prime}\cap M^{}_{n}(L^{\infty}(\mu))$ contains a finite
frame. For this purpose, we need to introduce the following
definition.

\begin{definition}
For an operator $A=(A^{}_{ij})^{}_{1\leq i,j\leq n}$ in
$M^{}_{n}(L^{\infty}(\mu))$, the \textit{$k$-diagonal} entries are
those $A^{}_{ij}$ with $j=i+k$.
\end{definition}

If $\{A\}^{\prime}\cap M^{}_{n}(L^{\infty}(\mu))$ contains a finite
frame $\{P^{}_{k}\}^{m}_{k=1}$, then we observe that every
$\mathrm{Tr}(P^{}_{k})$ is a $\mu$-measurable simple function on
$\Lambda$ for $k=1,\ldots,m$. Thus there exists a finite Borel
partition $\{\Lambda^{}_{j}\}^{r}_{j=1}$ of $\Lambda$ with respect
to $\{\mathrm{Tr}(P^{}_{k})\}^{m}_{k=1}$ such that every
$\mathrm{Tr}(P^{}_{k})$ takes a constant a.e.~$[\mu]$ on every
$\Lambda^{}_{j}$ for $j=1,\ldots,r$. Therefore, we can focus on the
local structures of $A$ restricted on each
$(L^{2}(\Lambda^{}_{j},\mu))^{(n)}$ respectively, and then combine
them together. In this sense, we assume that every
$\mathrm{Tr}(P^{}_{k})$ takes a constant a.e.~$[\mu]$ on $\Lambda$
for $k=1,2,\ldots,m$, and we obtain the following proposition.

\begin{proposition}
Let $A$ be an operator in $M^{}_{n}(L^{\infty}(\mu))$ and
$\{A\}^{\prime}\cap M^{}_{n}(L^{\infty}(\mu))$ contains a finite
frame $\{P^{}_{k}\}^{m}_{k=1}$ such that every
$\mathrm{Tr}(P^{}_{k})$ takes a constant a.e.~$[\mu]$ on $\Lambda$.
Then there exists an invertible operator $X$ in
$M^{}_{n}(L^{\infty}(\mu))$ such that $XAX^{-1}$ is of the block
diagonal form
$$XAX^{-1}=\begin{pmatrix}
A^{}_{{1}}&\cdots&0\\
\vdots&\ddots&\vdots\\
0&\cdots&A^{}_{{m}}\\
\end{pmatrix}^{}_{m\times m} \eqno{(2.30)}$$  and $A^{}_{{k}}$ in
$M^{}_{n^{}_{k}}(L^{\infty}(\mu))$ is of the upper triangular form
$$A^{}_{{k}}=\begin{pmatrix}
A^{{k}}_{11}&A^{{k}}_{12}&\cdots&A^{{k}}_{1,n^{}_{k}}\\
0&A^{{k}}_{22}&\cdots&A^{{k}}_{2,n^{}_{k}}\\
\vdots&\vdots&\ddots&\vdots\\
0&0&\cdots&A^{{k}}_{n^{}_{k}n^{}_{k}}\\
\end{pmatrix}^{}_{n^{}_{k}\times n^{}_{k}}
\begin{matrix}
L^{2}(\mu)\\
L^{2}(\mu)\\
\vdots\\
L^{2}(\mu)\\
\end{matrix} \eqno{(2.31)}$$ such that:
\begin{enumerate}
\item the equality $\sum^{m}_{k=1}n^{}_{k}=n$ holds;
\item the equality $A^{{k}}_{jj}=A^{{k}}_{11}$ holds for
$j=2,\ldots,n^{}_{k}$ and $k=1,\ldots,m$;
\item the support of
$A^{{k}}_{j,j+1}$ equals $\Lambda$ for $j=1,\ldots,n^{}_{k}-1$ and
$k=1,\ldots,m$.
\end{enumerate}
\end{proposition}

We prove Proposition $2.8$ by the following two lemmas.

\begin{lemma}
Let $A^{}_{{}}$ in $M^{}_{n^{}_{}}(L^{\infty}(\mu))$ be of the upper
triangular form
$$A^{}_{}=\begin{pmatrix}
f^{}_{11}&f^{}_{12}&\cdots&f^{}_{1n}\\
0&f^{}_{22}&\cdots&f^{}_{2n}\\
\vdots&\vdots&\ddots&\vdots\\
0&0&\cdots&f^{}_{nn}\\
\end{pmatrix}
\begin{matrix}
L^{2}(\mu)\\
L^{2}(\mu)\\
\vdots\\
L^{2}(\mu)\\
\end{matrix}. \eqno{(2.32)}$$
If a finite frame of $\{A^{}_{{}}\}^{\prime}_{}\cap
M^{}_{n^{}_{}}(L^{\infty}(\mu))$ only contains the identity of
$M^{}_{n^{}_{}}(L^{\infty}(\mu))$, then the main diagonal entries of
$A$ are the same.
\end{lemma}

\begin{proof}
If it is not true, then there exist an $\epsilon>0$ and a Borel
subset $\Lambda^{}_{\epsilon}$ of $\Lambda$ with
$\mu(\Lambda^{}_{\epsilon})>0$ such that either the inequality
$|f^{}_{ii}(\lambda)-f^{}_{11}(\lambda)|\geq\epsilon$ holds for
every $\lambda$ in $\Lambda^{}_{\epsilon}$ or the equality
$f^{}_{ii}(\lambda)=f^{}_{11}(\lambda)$ holds for every $\lambda$ in
$\Lambda^{}_{\epsilon}$ where $i=1,2,\ldots,n$. We can choose Borel
representatives to fulfill the relations for every $\lambda$ in
$\Lambda^{}_{\epsilon}$.

We construct a nontrivial idempotent in $\{A^{}_{}\}^{\prime}\cap
M^{}_{n}(L^{\infty}(\mu))$ to draw a contradiction. For this
purpose, we switch entries in the main diagonal of $A$ by similar
transformations to simplify the computation. In the discussion that
follows, we restrict $A^{}_{}$ in
$M^{}_{n}(L^{\infty}(\Lambda^{}_{\epsilon},\mu))$.

If $f^{}_{k+1,k+1}=f^{}_{11}$ on $\Lambda^{}_{\epsilon}$ and
$|f^{}_{kk}(\lambda)-f^{}_{11}(\lambda)|\geq\epsilon$ holds for
every $\lambda$ in $\Lambda^{}_{\epsilon}$, then we obtain the
equality for $\phi=f^{}_{k,k+1}/(f^{}_{kk}-f^{}_{k+1,k+1})$:
$$\begin{pmatrix}
1&\phi\\
0&1\\
\end{pmatrix}
\begin{pmatrix}
f^{}_{kk}&f^{}_{k,k+1}\\
0&f^{}_{k+1,k+1}\\
\end{pmatrix}=
\begin{pmatrix}
f^{}_{kk}&0\\
0&f^{}_{k+1,k+1}\\
\end{pmatrix}
\begin{pmatrix}
1&\phi\\
0&1\\
\end{pmatrix}.
\eqno{(2.33)}$$ Then we can switch the two main diagonal entries of
$A$ by a similarity transformation:
$$\begin{pmatrix}
0&1\\
1&0\\
\end{pmatrix}
\begin{pmatrix}
f^{}_{kk}&0\\
0&f^{}_{k+1,k+1}\\
\end{pmatrix}
\begin{pmatrix}
0&1\\
1&0\\
\end{pmatrix}=
\begin{pmatrix}
f^{}_{k+1,k+1}&0\\
0&f^{}_{kk}\\
\end{pmatrix}. \eqno{(2.34)}$$
We construct an invertible operator
$T=I^{}_{M^{}_{k-1}(L^{\infty}(\Lambda^{}_{\epsilon},\mu))}\oplus
S^{}_{2}\oplus
I^{}_{M^{}_{n-k-1}(L^{\infty}(\Lambda^{}_{\epsilon},\mu))}$ in
$M^{}_{n}(L^{\infty}(\Lambda^{}_{\epsilon},\mu))$, where
$I^{}_{M^{}_{k}(L^{\infty}(\Lambda^{}_{\epsilon},\mu))}$ is the
identity of $M^{}_{k}(L^{\infty}(\Lambda^{}_{\epsilon},\mu))$ and
$S^{}_{2}$ in $M^{}_{2}(L^{\infty}(\Lambda^{}_{\epsilon},\mu))$ is
of the form
$$S^{}_{2}=\begin{pmatrix}
0&1\\
1&0\\
\end{pmatrix}
\begin{pmatrix}
1&f^{}_{k,k+1}/(f^{}_{kk}-f^{}_{k+1,k+1})\\
0&1\\
\end{pmatrix}.  \eqno{(2.35)}$$
By the construction of $T$, we obtain that $TA^{}_{}T^{-1}$ is of
the upper triangular form and the $(k,k)$ and $(k+1,k+1)$ entries of
$A^{}_{}$ are switched to the $(k+1,k+1)$ and $(k,k)$ entries of
$TA^{}_{}T^{-1}$ respectively. Iterating this construction and
making the related similar transformation to $A^{}_{}$ one by one
for finite steps, we obtain an upper triangular matrix $A^{}_{1}$ in
$M^{}_{n}(L^{\infty}(\Lambda^{}_{\epsilon},\mu))$ of the form
$$A^{}_{1}=\begin{pmatrix}
h^{}_{11}&h^{}_{12}&\cdots&h^{}_{1n}\\
0&h^{}_{22}&\cdots&h^{}_{2n}\\
\vdots&\vdots&\ddots&\vdots\\
0&0&\cdots&h^{}_{nn}\\
\end{pmatrix}
\begin{matrix}
L^{2}(\Lambda^{}_{\epsilon},\mu)\\
L^{2}(\Lambda^{}_{\epsilon},\mu)\\
\vdots\\
L^{2}(\Lambda^{}_{\epsilon},\mu)\\
\end{matrix} \eqno{(2.36)}$$ such that the first $r$ main diagonal
entries of $A^{}_{1}$ are equal to $f^{}_{11}$ on
$\Lambda^{}_{\epsilon}$ and the equality
$|h^{}_{kk}(\lambda)-h^{}_{11}(\lambda)|\geq\epsilon$ holds for
every $\lambda$ in $\Lambda^{}_{\epsilon}$ for $k=r+1,r+2,\ldots,n$.

With the preceding preparation, we construct a nontrivial idempotent
$P$ in $\{A^{}_{1}\}^{\prime}\cap
M^{}_{n}(L^{\infty}(\Lambda^{}_{\epsilon},\mu))$ of the form
$$P=\begin{pmatrix}
I^{}_{M^{}_{r}(L^{\infty}(\Lambda^{}_{\epsilon},\mu))}&R\\
0&0\\
\end{pmatrix}. \eqno{(2.37)}$$ It is sufficient to ensure the
existence of $R$ in $M^{}_{r\times
(n-r)}(L^{\infty}(\Lambda^{}_{\epsilon},\mu))$. Let $R$ be of the
form
$$R=\begin{pmatrix}
\phi^{}_{1,r+1}&\phi^{}_{1,r+2}&\cdots&\phi^{}_{1n}\\
\phi^{}_{2,r+1}&\phi^{}_{2,r+2}&\cdots&\phi^{}_{2n}\\
\vdots&\vdots&\ddots&\vdots\\
\phi^{}_{r,r+1}&\phi^{}_{r,r+2}&\cdots&\phi^{}_{rn}\\
\end{pmatrix}^{}_{r\times
(n-r)}
\begin{matrix}
L^{2}(\Lambda^{}_{\epsilon},\mu)\\
L^{2}(\Lambda^{}_{\epsilon},\mu)\\
\vdots\\
L^{2}(\Lambda^{}_{\epsilon},\mu)\\
\end{matrix}. \eqno{(2.38)}$$ If the equality $A^{}_{1}P=PA^{}_{1}$
holds for the $(r,r+1)$ entry of $A^{}_{1}P$, we obtain the equality
$$h^{}_{rr}\phi^{}_{r,r+1}=h^{}_{r,r+1}+\phi^{}_{r,r+1}h^{}_{r+1,r+1}.
\eqno{(2.39)}$$ Thus we obtain
$$\phi^{}_{r,r+1}=h^{}_{r,r+1}/(h^{}_{rr}-h^{}_{r+1,r+1}).
\eqno{(2.40)}$$ For the $(r-1,r+1)$ entry of $A^{}_{1}P$, we obtain
the equality
$$h^{}_{r-1,r-1}\phi^{}_{r-1,r+1}+h^{}_{r-1,r}\phi^{}_{r,r+1}
=h^{}_{r-1,r+1}+\phi^{}_{r-1,r+1}h^{}_{r+1,r+1}. \eqno{(2.41)}$$
Thus we obtain
$$\phi^{}_{r-1,r+1}
=(h^{}_{r-1,r+1}-h^{}_{r-1,r}\phi^{}_{r,r+1})/(h^{}_{r-1,r-1}-h^{}_{r+1,r+1}).
\eqno{(2.42)}$$ In this way, we obtain $\phi^{}_{k,r+1}$ one by one
for $k=r,r-1,\ldots,1$. For the $(r,r+2)$ entry of $A^{}_{1}P$, we
obtain the equality
$$h^{}_{rr}\phi^{}_{r,r+2}
=h^{}_{r,r+2}+\phi^{}_{r,r+1}h^{}_{r+1,r+2}+\phi^{}_{r,r+2}h^{}_{r+2,r+2}.
\eqno{(2.43)}$$ Thus we obtain
$$\phi^{}_{r,r+2}
=(h^{}_{r,r+2}+\phi^{}_{r,r+1}h^{}_{r+1,r+2})/(h^{}_{rr}-h^{}_{r+2,r+2}).
\eqno{(2.44)}$$ In this way and by entries in column $(r+1)$ of $P$,
we obtain $\phi^{}_{k,r+2}$ one by one for $k=r,r-1,\ldots,1$.
Similarly, we obtain the left columns of $P$ one after another. And
this $P$ does exist in $\{A^{}_{1}\}^{\prime}\cap
M^{}_{n}(L^{\infty}(\Lambda^{}_{\epsilon},\mu))$. Notice that
$n>\mbox{Tr}(P^{}_{})(\lambda)=r^{}_{}>0$ (a constant) on
$\Lambda^{}_{\epsilon}$ and $A^{}_{1}$ is similar to $A$ in
$M^{}_{n}(L^{\infty}(\Lambda^{}_{\epsilon},\mu))$. Therefore there
exists an idempotent $Q$ in $\{A^{}_{}\}^{\prime}\cap
M^{}_{n}(L^{\infty}(\Lambda^{}_{\epsilon},\mu))$ such that
$n>\mbox{Tr}(Q^{}_{})(\lambda)=r^{}_{}>0$ (a constant) on
$\Lambda^{}_{\epsilon}$. Let $\mbox{Tr}(Q^{}_{})(\lambda)=n^{}_{}$
on $\Lambda\backslash\Lambda^{}_{\epsilon}$. Then this is a
contradiction between the existence of $Q$ and the assumption that
the frame only contains the identity of $M^{}_{n}(L^{\infty}(\mu))$.
Hence the equality $f^{}_{11}=f^{}_{ii}$ holds for $i=2,\ldots,n$.
\end{proof}

\begin{lemma}
Let $A^{}_{{}}$ in $M^{}_{n^{}_{}}(L^{\infty}(\mu))$ be of the upper
triangular form as in ${(2.32)}$. If a finite frame of
$\{A^{}_{{}}\}^{\prime}_{}\cap M^{}_{n^{}_{}}(L^{\infty}(\mu))$ only
contains the identity of $M^{}_{n^{}_{}}(L^{\infty}(\mu))$, then the
support of $f^{}_{i,i+1}$ is $\Lambda$ for $i=1,2,\ldots,n-1$.
\end{lemma}

\begin{proof}
If the lemma is not true, then there exist a positive integer $r$ in
$\{2,\ldots,n\}$ and a Borel subset $\Lambda^{}_{0}$ of $\Lambda$
with $\mu(\Lambda^{}_{0})>0$ such that:
\begin{enumerate}
\item $f^{}_{r-1,r}=0$ on $\Lambda^{}_{0}$;
\item the support of $f^{}_{i,i+1}$ contains $\Lambda^{}_{0}$ for
$i=1,\ldots,r-2$ if $r\geq 3$.
\end{enumerate}
Without loss of generality, we assume that:
\begin{enumerate}
\item[($1^{\prime}$)] $f^{}_{r-1,r}=0$ on $\Lambda^{}_{}$;
\item[($2^{\prime}$)] there exists an $\epsilon>0$ such that the inequality
$$|f^{}_{i,i+1}(\lambda)|\geq\epsilon \eqno{(2.45)}$$ holds for every $\lambda$ in
$\Lambda^{}_{}$ and $i=1,\ldots,r-2$ if $r\geq 3$.
\end{enumerate}

In the following, we show that $A$ is similar to a block diagonal
matrix in $M^{}_{n^{}_{}}(L^{\infty}(\mu))$ by reduction. Then we
obtain a contradiction.

By the preceding lemma and the assumption in $(2.45)$, we construct
an invertible operator $S^{\prime}_{r}$ in
$M^{}_{r}(L^{\infty}(\mu))$ such that the equality
$$\begin{pmatrix}
f^{}_{11}&\cdots&f^{}_{1,r-1}&f^{}_{1r}\\
\vdots&\ddots&\vdots&\vdots\\
0&\cdots&f^{}_{r-1,r-1}&0\\
0&\cdots&0&f^{}_{rr}\\
\end{pmatrix}S^{\prime}_{r}
=S^{\prime}_{r}\begin{pmatrix}
f^{}_{11}&\cdots&f^{}_{1,r-1}&0\\
\vdots&\ddots&\vdots&\vdots\\
0&\cdots&f^{}_{r-1,r-1}&0\\
0&\cdots&0&f^{}_{rr}\\
\end{pmatrix} \eqno{(2.46)}$$ holds.
Let $S^{\prime}_{r}$ be of the form
$$S^{\prime}_{r}=\begin{pmatrix}
1&0&\cdots&0&\phi^{}_{1r}\\
0&1&\cdots&0&\phi^{}_{2r}\\
\vdots&\vdots&\ddots&\vdots&\vdots\\
0&0&\cdots&1&\phi^{}_{r-1,r}\\
0&0&\cdots&0&1\\
\end{pmatrix}^{}_{r\times r}.  \eqno{(2.47)}$$

Note that $f^{}_{11}=f^{}_{jj}$ for $j=2,\ldots,n$. With respect to
the $(r-2,r)$ entries of the products on both sides of the equality
in $(2.46)$, the equality
$$f^{}_{r-2,r-2}\phi^{}_{r-2,r}+f^{}_{r-2,r-1}\phi^{}_{r-1,r}+f^{}_{r-2,r}=
\phi^{}_{r-2,r}f^{}_{rr} \eqno{(2.48)}$$ yields that
$\phi^{}_{r-1,r}=-f^{}_{r-2,r}\slash f^{}_{r-2,r-1}$. Thus with
respect to the $(r-3,r)$ entries of the products on both sides of
the equality in $(2.46)$, the equality
$$f^{}_{r-3,r-3}\phi^{}_{r-3,r}+f^{}_{r-3,r-2}\phi^{}_{r-2,r}+f^{}_{r-3,r-1}\phi^{}_{r-1,r}+f^{}_{r-3,r}=
\phi^{}_{r-3,r}f^{}_{rr} \eqno{(2.49)}$$ yields that
$\phi^{}_{r-2,r}=-(f^{}_{r-3,r-1}\phi^{}_{r-1,r}+f^{}_{r-3,r})\slash
f^{}_{r-3,r-2}$. In this way, we obtain $\phi^{}_{ir}$ for
$i=r-1,r-2,\ldots,2$. We can choose $0$ to be $\phi^{}_{1r}$. The
existence of $S^{\prime}_{r}$ is proved.

Note that $S^{}_{r}=S^{\prime}_{r}\oplus
I^{}_{M^{}_{n-r}(L^{\infty}(\mu))}$ is invertible in
$M^{}_{n}(L^{\infty}(\mu))$. The first $r$ rows and $r$ columns of
$S^{}_{r}A^{}_{}S^{-1}_{r}$ form an $r$-by-$r$ square matrix which
is block diagonal. Denote by $A^{\prime}_{r+1}$ the
$(r+1)$-by-$(r+1)$ square matrix formed by the first $(r+1)$ rows
and $(r+1)$ columns of $S^{}_{r}A^{}_{}S^{-1}_{r}$. Thus
$A^{\prime}_{r+1}$ is of the block form
$$A^{\prime}_{r+1}=\begin{pmatrix}
J^{}_{{1}}&0&R^{}_{{1}}\\
0&J^{}_{{2}}&R^{}_{{2}}\\
0&0&f^{}_{r+1,r+1}\\
\end{pmatrix} \eqno{(2.50)}$$ where $J^{}_{{1}}$ is the
$(r-1)$-by-$(r-1)$ square matrix formed by the first $(r-1)$ rows
and $(r-1)$ columns of $S^{}_{r}A^{}_{}S^{-1}_{r}$ and
$J^{}_{{2}}=f^{}_{rr}$. We show that there exists a Borel subset
$\Lambda^{}_{1}$ of $\Lambda$ with $\mu(\Lambda^{}_{1})>0$ such that
$A^{\prime}_{r+1}$ is similar to a block diagonal form restricted in
$M^{}_{r+1}(L^{\infty}(\Lambda^{}_{1},\mu))$.

If $R^{}_{2}=0$ on $\Lambda$, then $A^{\prime}_{r+1}$ is unitarily
equivalent to the block form
$$A^{\prime}_{r+1}\cong\begin{pmatrix}
J^{}_{{2}}&0&0\\
0&J^{}_{{1}}&R^{}_{{1}}\\
0&0&f^{}_{r+1,r+1}\\
\end{pmatrix}. \eqno{(2.51)}$$ Thus if the $(r-1,r+1)$ entry
$f^{\prime}_{r-1,r+1}$ of $S^{}_{r}A^{}_{}S^{-1}_{r}$ is $0$, then a
similar computation as for $S^{\prime}_{r}$ in $(2.46)$ and $(2.47)$
yields that the operator $A^{\prime}_{r+1}$ is similar to a block
diagonal matrix of the form
$$A^{\prime}_{r+1}\sim\begin{pmatrix}
J^{}_{{2}}&0&0\\
0&J^{}_{{1}}&0\\
0&0&f^{}_{r+1,r+1}\\
\end{pmatrix}. \eqno{(2.52)}$$ If $f^{\prime}_{r-1,r+1}\neq 0$, then there
exists a Borel subset $\Lambda^{}_{1}$ of $\Lambda$ with
$\mu(\Lambda^{}_{1})>0$ such that the inequality
$|f^{\prime}_{r-1,r+1}(\lambda)|\geq\epsilon^{}_{1}$ holds for some
$\epsilon^{}_{1}>0$ and every $\lambda$ in $\Lambda^{}_{1}$. Thus
let $J^{\prime}_{1}$ be the operator restricted in
$M^{}_{r}(L^{\infty}(\Lambda^{}_{1},\mu))$ of the form
$$J^{\prime}_{1}=\begin{pmatrix}
J^{}_{{1}}&R^{}_{{1}}\\
0&f^{}_{r+1,r+1}\\
\end{pmatrix}. \eqno{(2.53)}$$ Therefore $J^{\prime}_{1}$ and
$J^{}_{2}$ restricted in $M^{}_{r}(L^{\infty}(\Lambda^{}_{1},\mu))$
and $L^{\infty}(\Lambda^{}_{1},\mu)$ respectively are the diagonal
blocks we need.

On the other hand, if $R^{}_{2}\neq 0$ on $\Lambda$, then there
exists a Borel subset $\Lambda^{}_{1}$ of $\Lambda$ with
$\mu(\Lambda^{}_{1})>0$ such that the inequality
$|R^{}_{2}(\lambda)|\geq\epsilon^{}_{1}$ holds for some
$\epsilon^{}_{1}>0$ and every $\lambda$ in $\Lambda^{}_{1}$. We need
to consider two subcases. In one subcase, the $(r-1,r+1)$ entry
$f^{\prime}_{r-1,r+1}$ of $S^{}_{r}A^{}_{}S^{-1}_{r}$ is $0$ on
$\Lambda^{}_{1}$, then a computation as for $S^{\prime}_{r}$ in
$(2.46)$ and $(2.47)$ yields that $A^{\prime}_{r+1}$ is similar to a
block diagonal form restricted in
$M^{}_{r+1}(L^{\infty}(\Lambda^{}_{1},\mu))$
$$A^{\prime}_{r+1}\sim\begin{pmatrix}
J^{}_{{1}}&0&0\\
0&J^{}_{{2}}&R^{}_{2}\\
0&0&f^{}_{r+1,r+1}\\
\end{pmatrix}. \eqno{(2.54)}$$
Let $J^{\prime}_{2}$ be the operator restricted in
$M^{}_{2}(L^{\infty}(\Lambda^{}_{1},\mu))$ of the form
$$J^{\prime}_{2}=\begin{pmatrix}
J^{}_{{2}}&R^{}_{{2}}\\
0&f^{}_{r+1,r+1}\\
\end{pmatrix}. \eqno{(2.55)}$$ Therefore $J^{}_{1}$ and
$J^{\prime}_{2}$ restricted in
$M^{}_{r-1}(L^{\infty}(\Lambda^{}_{1},\mu))$ and
$M^{}_{2}(L^{\infty}(\Lambda^{}_{1},\mu))$ respectively are the
diagonal blocks we need.

In the other subcase, without loss of generality, we assume that the
inequality $|f^{\prime}_{r-1,r+1}(\lambda)|\geq \epsilon^{}_{2}$
holds a.e.~$[\mu]$ on $\Lambda^{}_{1}$ for the $(r-1,r+1)$ entry
$f^{\prime}_{r-1,r+1}$ of $S^{}_{r}A^{}_{}S^{-1}_{r}$. If $r=2$,
then $A^{\prime}_{r+1}$ is similar to a block diagonal form as in
$(2.54)$. If $r>2$, then $A^{\prime}_{r+1}$ is similar to a block
diagonal form as in $(2.51)$.

Let $S^{\prime}_{r+1}$ be the invertible operator in
$M^{}_{r+1}(L^{\infty}(\Lambda^{}_{1},\mu))$ such that
$S^{\prime}_{r+1}A^{\prime}_{r+1}S^{\prime-1}_{r+1}$ is of the block
diagonal form as discussed in the preceding paragraphs. The operator
$S^{}_{r+1}=S^{\prime}_{r+1}\oplus
I^{}_{M^{}_{n-r-1}(L^{\infty}(\Lambda^{}_{1},\mu))}$ is invertible
in $M^{}_{n}(L^{\infty}(\Lambda^{}_{1},\mu))$ and we write
$A^{}_{r+1}=S^{}_{r+1}S^{}_{r}A^{}_{}S^{-1}_{r}S^{-1}_{r+1}$. The
first $(r+1)$ rows and $(r+1)$ columns of $A^{}_{r+1}$ form an
$(r+1)$-by-$(r+1)$ square matrix which is block diagonal.

Denote by $A^{\prime}_{r+2}$ the $(r+2)$-by-$(r+2)$ square matrix
formed by the first $(r+2)$ rows and $(r+2)$ columns of
$A^{}_{r+1}$. Iterating the preceding discussion, there exists a
Borel subset $\Lambda^{}_{2}$ of $\Lambda^{}_{1}$ with
$\mu(\Lambda^{}_{2})>0$ and an invertible operator $S^{}_{r+2}$ in
$M^{}_{n}(L^{\infty}(\Lambda^{}_{2},\mu))$ such that the
$(r+2)$-by-$(r+2)$ square matrix formed by the first $(r+2)$ rows
and $(r+2)$ columns of $S^{}_{r+2}A^{}_{r+1}S^{-1}_{r+2}$ is of the
block diagonal form. Iterating this construction for finite steps,
we obtain a Borel subset $\Lambda^{}_{m}$ of $\Lambda^{}_{}$ with
$\mu(\Lambda^{}_{m})>0$ and an invertible operator $S^{}_{m}$ in
$M^{}_{n}(L^{\infty}(\Lambda^{}_{m},\mu))$ such that
$S^{}_{m}A^{}_{}S^{-1}_{m}$ is of the block diagonal form
$$S^{}_{m}A^{}_{0}S^{-1}_{m}=\begin{pmatrix}
J^{}_{r^{}_{1}}&0&\cdots&0\\
0&J^{}_{r^{}_{2}}&\cdots&0\\
\vdots&\vdots&\ddots&\vdots\\
0&0&\cdots&J^{}_{r^{}_{k}}\\
\end{pmatrix}, \eqno{(2.56)}$$ where $J^{}_{r^{}_{i}}$ in
$M^{}_{r^{}_{i}}(L^{\infty}(\Lambda^{}_{m},\mu))$ for
$i=1,2,\ldots,k$. By the construction of
$S^{}_{m}A^{}_{}S^{-1}_{m}$, we obtain that $1\leq r^{}_{1}<n$.
Therefore there exists a nontrivial idempotent $P$ in the relative
commutant $\{A^{}_{}\}^{\prime}\cap M^{}_{n}(L^{\infty}(\mu))$ such
that $0<\mathrm{Tr}(P)(\lambda)<n$ for every $\lambda^{}_{}$ in
$\Lambda^{}_{m}$ and $\mathrm{Tr}(P)(\lambda)=n$ for every
$\lambda^{}_{}$ in $\Lambda^{}_{}\backslash\Lambda^{}_{m}$. This is
a contradiction between the existence of $P$ and the assumption that
a frame in the relative commutant $\{A^{}_{}\}^{\prime}\cap
M^{}_{n}(L^{\infty}(\mu))$ only contains the identity of
$M^{}_{n}(L^{\infty}(\mu))$. Hence the proof is finished.
\end{proof}

\begin{proof}[Proof of Proposition $2.8$]
By Lemma $2.6$, there exists an invertible operator $Y$ in
$M^{}_{n}(L^{\infty}(\mu))$ such that every $YP^{}_{k}Y^{-1}$ is a
diagonal projection in $M^{}_{n}(L^{\infty}(\mu))$ for
$k=1,\ldots,n$. Note that $\{YP^{}_{k}Y^{-1}\}^{m}_{k=1}$ is a
finite frame in $\{YAY^{-1}\}^{\prime}\cap
M^{}_{n}(L^{\infty}(\mu))$.

If $\{YP^{}_{k}Y^{-1}\}^{m}_{k=1}$ contains only one element, the
identity of $M^{}_{n}(L^{\infty}(\mu))$, then this frame and
$\mathscr{E}^{}_{n}$, the set of central projections in
$M^{}_{n}(L^{\infty}(\mu))$, generate a bounded maximal abelian set
of idempotents in $\{YAY^{-1}\}^{\prime}\cap
M^{}_{n}(L^{\infty}(\mu))$ which is also $\mathscr{E}^{}_{n}$.

By (\cite{Deckard_0}, Theorem $2$) and for the sake of simplicity,
there exists a unitary operator $U$ in $M^{}_{n}(L^{\infty}(\mu))$
such that $UYAY^{-1}U^{*}(=A^{}_{0})$ is of the upper triangular
form as in $(2.32)$. By Lemma $2.9$ and Lemma $2.10$, we obtain that
the entries in the main diagonal of $A^{}_{0}$ are the same and the
$1$-diagonal entries of $A^{}_{0}$ are supported on $\Lambda$.

If $\{YP^{}_{k}Y^{-1}\}^{m}_{k=1}$ contains more than one element,
then $YAY^{-1}$ can be expressed in the diagonal form
$$YAY^{-1}=\begin{pmatrix}
B^{}_{{1}}&\cdots&0\\
\vdots&\ddots&\vdots\\
0&\cdots&B^{}_{{m}}\\
\end{pmatrix}^{}_{m\times m}, \eqno{(2.57)}$$ where every $B^{}_{{k}}$ is
the restriction of $YAY^{-1}$ on the range of $YP^{}_{k}Y^{-1}$ for
$k=1,\ldots,m$. Write $\mathrm{Tr}(P^{}_{k})(\lambda)=n^{}_{k}$
a.e.~$[\mu]$ on $\Lambda$. Then the restriction of $YP^{}_{k}Y^{-1}$
on its range becomes the identity of
$M^{}_{n^{}_{k}}(L^{\infty}(\mu))$ and $B^{}_{{k}}$ is in
$M^{}_{n^{}_{k}}(L^{\infty}(\mu))$. In particular, the identity of
$M^{}_{n^{}_{k}}(L^{\infty}(\mu))$ forms a finite frame in
$\{B^{}_{k}\}^{\prime}_{}\cap M^{}_{n^{}_{k}}(L^{\infty}(\mu))$.
Therefore by (\cite{Deckard_0}, Theorem $2$), there exists a unitary
operator $U^{}_{k}$ in $M^{}_{n^{}_{k}}(L^{\infty}(\mu))$ such that
$U^{}_{k}B^{}_{k}U^{*}_{k}$ is of the upper triangular form
$$U^{}_{k}B^{}_{k}U^{*}_{k}=\begin{pmatrix}
A^{{k}}_{11}&A^{{k}}_{12}&\cdots&A^{{k}}_{1,n^{}_{k}}\\
0&A^{{k}}_{22}&\cdots&A^{{k}}_{2,n^{}_{k}}\\
\vdots&\vdots&\ddots&\vdots\\
0&0&\cdots&A^{{k}}_{n^{}_{k}n^{}_{k}}\\
\end{pmatrix}^{}_{n^{}_{k}\times n^{}_{k}}
\begin{matrix}
L^{2}(\mu)\\
L^{2}(\mu)\\
\vdots\\
L^{2}(\mu)\\
\end{matrix}. \eqno{(2.58)}$$ Write
$A^{}_{k}=U^{}_{k}B^{}_{k}U^{*}_{k}$. Then by the foregoing proof,
we obtain that $A^{}_{k}$ is with the required properties for
$k=1,\ldots,m$.
\end{proof}

Note that the reverse assertion is obvious by (\cite{Shi_2}, Lemma
$3.1$). Therefore for an operator $A$ in
$M^{}_{n^{}_{}}(L^{\infty}(\mu))$, we obtain a necessary and
sufficient condition to connect the structures of $A$ and a finite
frame of the relative commutant $\{A\}^{\prime}\cap
M^{}_{n^{}_{}}(L^{\infty}(\mu))$. We finish this section with an
example to show that there exists an operator in
$M^{}_{2^{}_{}}(L^{\infty}(\mu))$ such that the relative commutant
contains no finite frames.

\begin{example}
Let $A$ in $M^{}_{2^{}_{}}(L^{\infty}(\mu))$ be of the form
$$A=\begin{pmatrix}
f^{}_{}&\mathbf{1}\\
0&-2f\\
\end{pmatrix}
\begin{matrix}
L^{2}(\mu)\\
L^{2}(\mu)\\
\end{matrix}, \eqno{(2.59)}$$
where $f$ is injective and the spectrum of $f$ is the interval
$[0,1]$. If $P$ is an idempotent in the relative commutant
$\{A\}^{\prime}\cap M^{}_{2^{}_{}}(L^{\infty}(\mu))$, then we can
express $P$ in the form
$$P=\begin{pmatrix}
p^{}_{11}&p^{}_{12}\\
p^{}_{21}&p^{}_{22}\\
\end{pmatrix}
\begin{matrix}
L^{2}(\mu)\\
L^{2}(\mu)\\
\end{matrix}, \eqno{(2.60)}$$ where $p^{}_{ij}$ is in
$L^{\infty}(\mu)$ for $1\leq i,j\leq 2$. The equality $AP=PA$ yields
that
$$\begin{cases}
-2f^{}_{}p^{}_{22}=p^{}_{21}-2p^{}_{22}f^{}_{};\\
f^{}_{}p^{}_{12}+p^{}_{22}
=p^{}_{11}-2p^{}_{12}f^{}_{}.\\
\end{cases} \eqno{(2.61)}$$ Therefore, we obtain
$$\begin{cases}
p^{}_{21}=0;\\
3f^{}_{}p^{}_{12}
=p^{}_{11}-p^{}_{22}.\\
\end{cases} \eqno{(2.62)}$$ Since $P$ is an idempotent, we obtain
that $p^{}_{ii}$ is a characteristic function in $L^{\infty}(\mu)$
for $i=1,2$. A further computation shows that the idempotents in the
relative commutant $\{A\}^{\prime}\cap
M^{}_{2^{}_{}}(L^{\infty}(\mu))$ form an abelian set. Note that this
abelian set of idempotents is unbounded. Thus there does not exist
an idempotent $Q$ in $\{A\}^{\prime}\cap
M^{}_{2^{}_{}}(L^{\infty}(\mu))$ such that
$\mathrm{Tr}(Q)(\lambda)=1$ a.e.~$[\mu]$ on $\Lambda$. If
$\{A\}^{\prime}\cap M^{}_{2^{}_{}}(L^{\infty}(\mu))$ contains a
finite frame $\{P^{}_{k}\}^{m}_{k=1}$, then there exists a
$P^{}_{k^{}_{0}}$ in the finite frame such that
$f(\mathrm{supp}(\mathrm{Tr}(P^{}_{k^{}_{0}})))$ contains an open
neighborhood of $0$. Thus $\mathrm{Tr}(P^{}_{k^{}_{0}})$ can not
take a constant on its support and we can construct a proper
sub-idempotent of $P^{}_{k^{}_{0}}$ in $\{A\}^{\prime}\cap
M^{}_{2^{}_{}}(L^{\infty}(\mu))$ supported on
$\mathrm{supp}(\mathrm{Tr}(P^{}_{k^{}_{0}}))$. This is a
contradiction.
\end{example}

\section{Bounded maximal abelian sets of idempotents in
$\{A\}^{\prime}\cap M^{}_{n^{}_{}}(L^{\infty}(\mu))$}

For a matrix $A$ in $M^{}_{n}(\mathbb{C})$, any two (bounded)
maximal abelian sets of idempotents $\mathscr{P}$ and $\mathscr{Q}$
in $\{A\}^{\prime}\cap M^{}_{n}(\mathbb{C})$ are similar to each
other in $\{A\}^{\prime}\cap M^{}_{n}(\mathbb{C})$. In this section,
for $A$ in $M^{}_{n^{}_{}}(L^{\infty}(\mu))$ such that
$\{A\}^{\prime}\cap M^{}_{n^{}_{}}(L^{\infty}(\mu))$ contains a
finite frame, we prove that every two bounded maximal abelian sets
of idempotents $\mathscr{P}$ and $\mathscr{Q}$ in
$\{A\}^{\prime}\cap M^{}_{n^{}_{}}(L^{\infty}(\mu))$ are similar to
each other in $\{A\}^{\prime}\cap M^{}_{n^{}_{}}(L^{\infty}(\mu))$.

We need to mention that in (\cite{Shi_2}, Theorem $3.3$) we prove a
special case for the motivation mentioned in the preceding
paragraph. That is if an operator $A$ is as in the form of $(2.31)$
$$A^{}_{{}}=\begin{pmatrix}
A^{{}}_{11}&A^{{}}_{12}&\cdots&A^{{}}_{1n^{}_{}}\\
0&A^{{}}_{22}&\cdots&A^{{}}_{2n^{}_{}}\\
\vdots&\vdots&\ddots&\vdots\\
0&0&\cdots&A^{{}}_{n^{}_{}n^{}_{}}\\
\end{pmatrix}^{}_{n^{}_{}\times n^{}_{}}
\begin{matrix}
L^{2}(\mu)\\
L^{2}(\mu)\\
\vdots\\
L^{2}(\mu)\\
\end{matrix} \eqno{(3.1)}$$ such that:
\begin{enumerate}
\item $A^{{}}_{11}$ corresponds to an injective Borel function
$\phi$ in $L^{\infty}(\mu)$;
\item the equality $A^{{}}_{jj}=A^{{}}_{11}$ holds for $j=2,\ldots,n^{}_{}$;
\item the support of
$A^{{}}_{j,j+1}$ equals $\Lambda$ for $j=1,\ldots,n^{}_{}-1$,
\end{enumerate}
then every two bounded maximal abelian sets of idempotents
$\mathscr{P}$ and $\mathscr{Q}$ in the relative commutant
$\{A^{(k)}\}^{\prime}\cap M^{}_{nk^{}_{}}(L^{\infty}(\mu))$ are
similar to each other in $\{A^{(k)}\}^{\prime}\cap
M^{}_{nk^{}_{}}(L^{\infty}(\mu))$, where $k$ is a positive integer.
We generalize the models applied in (\cite{Shi_2}, Theorem $3.3$) to
be the forms characterized in $(2.30)$ and $(2.31)$ to continue our
study in this section.

Part of our main theorem is abstracted from the following example.

\begin{example}
Let $A$ in $M^{}_{4}(L^{\infty}(\mu))$ be of the form
$$A=\begin{pmatrix}
A^{}_{1}&{0}\\
{0}&A^{}_{2}\\
\end{pmatrix},\mbox{ and }
A^{}_{1}=\begin{pmatrix}
f&{f^{}_{1}}\\
{0}&f\\
\end{pmatrix}
\begin{matrix}
L^{2}(\mu)\\
L^{2}(\mu)\\
\end{matrix},\,
A^{}_{2}=\begin{pmatrix}
f&{f^{}_{2}}\\
{0}&f\\
\end{pmatrix}
\begin{matrix}
L^{2}(\mu)\\
L^{2}(\mu)\\
\end{matrix}, \eqno{(3.2)}$$
where $f^{}_{i}$ in $L^{\infty}(\mu)$ is supported on $\Lambda$ for
$i=1,2$. Then we prove that for every two bounded maximal abelian
sets of idempotents  $\mathscr{P}$ and $\mathscr{Q}$ in
$\{A\}^{\prime}\cap M^{}_{4^{}_{}}(L^{\infty}(\mu))$, there exists
an invertible operator $X$ in $\{A\}^{\prime}\cap
M^{}_{4^{}_{}}(L^{\infty}(\mu))$ such that
$X\mathscr{P}X^{-1}=\mathscr{Q}$. We divide the proof into two
parts.

First, we assert that for every idempotent $P$ in
$\{A\}^{\prime}\cap M^{}_{4^{}_{}}(L^{\infty}(\mu))$, there exists
an invertible operator $X$ in $\{A\}^{\prime}\cap
M^{}_{4^{}_{}}(L^{\infty}(\mu))$ such that $XPX^{-1}$ is diagonal.

For this purpose, we characterize the relative commutant of $A$ in
$M^{}_{4^{}_{}}(L^{\infty}(\mu))$. For an element $B$ in
$\{A\}^{\prime}\cap M^{}_{4^{}_{}}(L^{\infty}(\mu))$, we can write
$B$ in the block form
$$\begin{pmatrix}
B^{}_{11}&B^{}_{12}\\
B^{}_{21}&B^{}_{22}\\
\end{pmatrix} \eqno{(3.3)}$$ such that $B^{}_{ij}$ satisfies the equality
$A^{}_{i}B^{}_{ij}=B^{}_{ij}A^{}_{j}$ for $i,j=1,2$. By a similar
computation as in $(2.60)$, we obtain that $B^{}_{ij}$ can be
expressed in the form
$$B^{}_{ij}=\begin{pmatrix}
\varphi^{ij}_{11}&\varphi^{ij}_{12}\\
0&\varphi^{ij}_{22}\\
\end{pmatrix}
\begin{matrix}
L^{2}(\mu)\\
L^{2}(\mu)\\
\end{matrix}, \eqno{(3.4)}$$
where $\varphi^{ij}_{kl}$ is in $L^{\infty}(\mu)$ for $k,l=1,2$. The
equality $A^{}_{i}B^{}_{ii}=B^{}_{ii}A^{}_{i}$ yields that
$f^{}_{i}\varphi^{ii}_{22}=\varphi^{ii}_{11}f^{}_{i}$ for $i=1,2$.
By $f^{}_{i}(\lambda)\neq 0$ for almost every $\lambda$ in $\Lambda$
and $i=1,2$, we obtain that $\varphi^{ii}_{11}=\varphi^{ii}_{22}$
a.e.~$[\mu]$ on $\Lambda$ and $i=1,2$. Thus we abbreviate
$\varphi^{ii}_{jj}$ as $\varphi^{ii}_{}$ for $j=1,2$. The equality
$A^{}_{i}B^{}_{ij}=B^{}_{ij}A^{}_{j}$ yields that
$f^{}_{i}\varphi^{ij}_{22}=\varphi^{ij}_{11}f^{}_{j}$ for $i=1,2$
and $i\neq j$. Thus $B$ can be expressed in the form
$$B=\begin{pmatrix}
\varphi^{11}_{}&\varphi^{11}_{12}&\varphi^{12}_{11}&\varphi^{12}_{12}\\
0&\varphi^{11}_{}&0&\varphi^{12}_{22}\\
\varphi^{21}_{11}&\varphi^{21}_{12}&\varphi^{22}_{}&\varphi^{22}_{12}\\
0&\varphi^{21}_{22}&0&\varphi^{22}_{}\\
\end{pmatrix}
\begin{matrix}
L^{2}(\mu)\\
L^{2}(\mu)\\
L^{2}(\mu)\\
L^{2}(\mu)\\
\end{matrix}. \eqno{(3.5)}$$
Let $U\in M^{}_{4}(L^{\infty}(\mu))$ be of the form
$$U=\begin{pmatrix}
\mathbf{1}&0&0&0\\
0&0&\mathbf{1}&0\\
0&\mathbf{1}&0&0\\
0&0&0&\mathbf{1}\\
\end{pmatrix}
\begin{matrix}
L^{2}(\mu)\\
L^{2}(\mu)\\
L^{2}(\mu)\\
L^{2}(\mu)\\
\end{matrix}. \eqno{(3.6)}$$ Then $UBU^{*}\in
\{UAU^{*}\}^{\prime}\cap M^{}_{4^{}_{}}(L^{\infty}(\mu))$ is of the
form
$$UBU^{*}=\begin{pmatrix}
\varphi^{11}_{}&\varphi^{12}_{11}&\varphi^{11}_{12}&\varphi^{12}_{12}\\
\varphi^{21}_{11}&\varphi^{22}_{}&\varphi^{21}_{12}&\varphi^{22}_{12}\\
0&0&\varphi^{11}_{}&\varphi^{12}_{22}\\
0&0&\varphi^{21}_{22}&\varphi^{22}_{}\\
\end{pmatrix}
\begin{matrix}
L^{2}(\mu)\\
L^{2}(\mu)\\
L^{2}(\mu)\\
L^{2}(\mu)\\
\end{matrix}. \eqno{(3.7)}$$ Note that $\varphi^{ij}_{11}$ and
$\varphi^{ij}_{22}$ satisfy the equality
$f^{}_{i}\varphi^{ij}_{22}=\varphi^{ij}_{11}f^{}_{j}$ for $i\neq j$
and $i,j=1,2$.

Let $P$ be an idempotent in $\{UAU^{*}\}^{\prime}\cap
M^{}_{4^{}_{}}(L^{\infty}(\mu))$. Then by $(3.7)$ the idempotent $P$
can be expressed in the block form
$$P=\begin{pmatrix}
P^{}_{11}&P^{}_{12}\\
0&P^{}_{22}\\
\end{pmatrix},
P^{}_{11}=\begin{pmatrix}
\varphi^{11}_{}&\varphi^{12}_{11}\\
\varphi^{21}_{11}&\varphi^{22}_{}\\
\end{pmatrix},
P^{}_{22}=\begin{pmatrix}
\varphi^{11}_{}&\varphi^{12}_{22}\\
\varphi^{21}_{22}&\varphi^{22}_{}\\
\end{pmatrix}, \eqno{(3.8)}$$ where we obtain that $P^{}_{ii}$ is
an idempotent in $M^{}_{2}(L^{\infty}(\mu))$ for $i=1,2$. Let $R$ be
of the form
$$R=\begin{pmatrix}
0&P^{}_{12}\\
0&0\\
\end{pmatrix}. \eqno{(3.9)}$$ Then $R$ is in
$\{UAU^{*}\}^{\prime}\cap M^{}_{4^{}_{}}(L^{\infty}(\mu))$. Denote
by $\sigma^{}_{\{UAU^{*}\}^{\prime}\cap
M^{}_{4^{}_{}}(L^{\infty}(\mu))}(R)$ the spectrum of $R$ in the
unital Banach algebra $\{UAU^{*}\}^{\prime}\cap
M^{}_{4^{}_{}}(L^{\infty}(\mu))$. Thus the equality
$$\sigma^{}_{\{UAU^{*}\}^{\prime}\cap
M^{}_{4^{}_{}}(L^{\infty}(\mu))}(SR)=\sigma^{}_{\{UAU^{*}\}^{\prime}\cap
M^{}_{4^{}_{}}(L^{\infty}(\mu))}(RS)=\{0\} \eqno{(3.10)}$$ holds for
every operator $S$ in $\{UAU^{*}\}^{\prime}\cap
M^{}_{4^{}_{}}(L^{\infty}(\mu))$. Thus the equality
$$(2P-I)(2P-I-R)=I-(2P-1)R \eqno{(3.11)}$$ yields that $2P-I-R$ is
invertible in $\{UAU^{*}\}^{\prime}\cap
M^{}_{4^{}_{}}(L^{\infty}(\mu))$. (Note that $2P-I$ is invertible in
$\{UAU^{*}\}^{\prime}\cap M^{}_{4^{}_{}}(L^{\infty}(\mu))$.) Since
the equality $(P-R)^{2}=P+R^{2}-RP-PR=P-R$ yields that
$R^{2}-RP-PR+R=0$, we obtain
$$\begin{array}{rcl}
(P-R)(2P-I-R)&=&P-PR-2RP+R+R^{2}\\
&=&P-RP\\
&=&(2P-I-R)P\\
\end{array}, \eqno{(3.12)}$$ which means $P$ is similar to $P-R$ in
$\{UAU^{*}\}^{\prime}\cap M^{}_{4^{}_{}}(L^{\infty}(\mu))$. The
idempotent $P-R$ is of the form
$$P-R=\begin{pmatrix}
\varphi^{11}_{}&\varphi^{12}_{11}&0&0\\
\varphi^{21}_{11}&\varphi^{22}_{}&0&0\\
0&0&\varphi^{11}_{}&\varphi^{12}_{22}\\
0&0&\varphi^{21}_{22}&\varphi^{22}_{}\\
\end{pmatrix}
\begin{matrix}
L^{2}(\mu)\\
L^{2}(\mu)\\
L^{2}(\mu)\\
L^{2}(\mu)\\
\end{matrix}. \eqno{(3.13)}$$

Next, we need to construct an invertible operator in
$\{UAU^{*}\}^{\prime}\cap M^{}_{4^{}_{}}(L^{\infty}(\mu))$ such that
$P-R$ is similar to a diagonal projection.

Without loss of generality, we assume that
$\mathrm{Tr}(P^{}_{11})(\lambda)=1$ for every $\lambda$ in
$\Lambda$. By the proof of Lemma $2.3$, there exists a unitary
operator $V$ in $M^{}_{4}(L^{\infty}(\mu))$ such that $V(P-R)V^{*}$
is of the form
$$V(P-R)V^{*}=\begin{pmatrix}
\psi^{11}_{}&\psi^{12}_{11}&0&0\\
\psi^{21}_{11}&\psi^{22}_{}&0&0\\
0&0&\psi^{11}_{}&\psi^{12}_{22}\\
0&0&\psi^{21}_{22}&\psi^{22}_{}\\
\end{pmatrix}
\begin{matrix}
L^{2}(\mu)\\
L^{2}(\mu)\\
L^{2}(\mu)\\
L^{2}(\mu)\\
\end{matrix}, \eqno{(3.14)}$$
where $|\psi^{11}_{}(\lambda)|\geq 1\slash 2$ for every $\lambda$ in
$\Lambda$. Let $Y$ be in the form
$$Y=\begin{pmatrix}
\mathbf{1}&0&0&0\\
-{\psi^{21}_{11}}\slash{\psi^{11}_{}}&\mathbf{1}&0&0\\
0&0&\mathbf{1}&0\\
0&0&-{\psi^{21}_{22}}\slash{\psi^{11}_{}}&\mathbf{1}\\
\end{pmatrix}
\begin{matrix}
L^{2}(\mu)\\
L^{2}(\mu)\\
L^{2}(\mu)\\
L^{2}(\mu)\\
\end{matrix}. \eqno{(3.15)}$$
Then by Lemma $2.3$, we obtain that $YV(P-R)V^{*}Y^{-1}$ is diagonal
in $M^{}_{4}(L^{\infty}(\mu))$. By the constructions of $V$ and $Y$,
the operator $V^{*}YV$ is invertible in $\{UAU^{*}\}^{\prime}\cap
M^{}_{4^{}_{}}(L^{\infty}(\mu))$. Thus $V^{*}YV(2P-I-R)$ is an
invertible operator in $\{UAU^{*}\}^{\prime}\cap
M^{}_{4^{}_{}}(L^{\infty}(\mu))$. Therefore $U^{*}V^{*}YV(2P-I-R)U$
is an invertible operator in $\{A\}^{\prime}\cap
M^{}_{4^{}_{}}(L^{\infty}(\mu))$ as required and we achieve the
assertion.

Let $\mathscr{E}$ be the set of diagonal projections in
$\{A\}^{\prime}\cap M^{}_{4^{}_{}}(L^{\infty}(\mu))$. Then by a
similar computation as in $(2.61)$ and $(2.62)$, we can verify that
$\mathscr{E}$ is a bounded maximal abelian set of idempotents in
$\{A\}^{\prime}\cap M^{}_{4^{}_{}}(L^{\infty}(\mu))$. Assume that
$\mathscr{P}$ is a bounded maximal abelian set of idempotents in
$\{A\}^{\prime}\cap M^{}_{4^{}_{}}(L^{\infty}(\mu))$. Then by
Proposition $2.5$, $\{A\}^{\prime}\cap
M^{}_{4^{}_{}}(L^{\infty}(\mu))$ contains a finite frame
$\mathscr{P}^{}_{0}$. By the preceding assertion there exists an
invertible operator $X$ in $\{A\}^{\prime}\cap
M^{}_{4^{}_{}}(L^{\infty}(\mu))$ such that every element in
$X\mathscr{P}^{}_{0}X^{-1}$ is diagonal. Thus
$X\mathscr{P}^{}_{}X^{-1}$ is included in $\mathscr{E}$. The
maximality of $X\mathscr{P}^{}_{}X^{-1}$ in $\{A\}^{\prime}\cap
M^{}_{4^{}_{}}(L^{\infty}(\mu))$ ensures that
$X\mathscr{P}^{}_{}X^{-1}=\mathscr{E}$. Therefore, every bounded
maximal abelian set of idempotents is similar to $\mathscr{E}$ in
$\{A\}^{\prime}\cap M^{}_{4^{}_{}}(L^{\infty}(\mu))$. This example
is finished. \qed
\end{example}

To generalize Example $(3.1)$, we need the following lemmas.

\begin{lemma}
Let $A$ in $M^{}_{n^{}_{}}(L^{\infty}(\mu))$ be of the upper
triangular form
$$A^{}_{{}}=\begin{pmatrix}
A^{{}}_{11}&A^{{}}_{12}&\cdots&A^{{}}_{1n^{}_{}}\\
0&A^{{}}_{22}&\cdots&A^{{}}_{2n^{}_{}}\\
\vdots&\vdots&\ddots&\vdots\\
0&0&\cdots&A^{{}}_{n^{}_{}n^{}_{}}\\
\end{pmatrix}^{}_{n^{}_{}\times n^{}_{}}
\begin{matrix}
L^{2}(\mu)\\
L^{2}(\mu)\\
\vdots\\
L^{2}(\mu)\\
\end{matrix} \eqno{(3.16)}$$
such that
\begin{enumerate}
\item $A^{{}}_{ii}=A^{{}}_{jj}$ for $i,j=1,2,\ldots,n$;
\item $A^{{}}_{i,i+1}$ is supported on
$\Lambda$ for $i=1,2,\ldots,n-1$.
\end{enumerate}
If $B$ is an operator in $\{A\}^{\prime}\cap
M^{}_{n^{}_{}}(L^{\infty}(\mu))$, then $B$ is of the form
$$B^{}_{{}}=\begin{pmatrix}
B^{{}}_{11}&B^{{}}_{12}&\cdots&B^{{}}_{1n^{}_{}}\\
0&B^{{}}_{22}&\cdots&B^{{}}_{2n^{}_{}}\\
\vdots&\vdots&\ddots&\vdots\\
0&0&\cdots&B^{{}}_{n^{}_{}n^{}_{}}\\
\end{pmatrix}^{}_{n^{}_{}\times n^{}_{}}
\begin{matrix}
L^{2}(\mu)\\
L^{2}(\mu)\\
\vdots\\
L^{2}(\mu)\\
\end{matrix} \eqno{(3.17)}$$ such that $B^{{}}_{ii}=B^{{}}_{jj}$ for
$i,j=1,2,\ldots,n$.
\end{lemma}

\begin{proof}
For the sake of simplicity, we use the relaxed convention of
treating Borel representatives as elements in $L^{\infty}(\mu)$ and
consider $A^{}_{ij}$ and $B^{}_{ij}$ as elements in
$L^{\infty}(\mu)$ for $i,j=1,\ldots,n$. Let $B$ in
$\{A\}^{\prime}\cap M^{}_{n^{}_{}}(L^{\infty}(\mu))$ be of the form
$$B^{}_{{}}=\begin{pmatrix}
B^{{}}_{11}&B^{{}}_{12}&\cdots&B^{{}}_{1n^{}_{}}\\
B^{{}}_{21}&B^{{}}_{22}&\cdots&B^{{}}_{2n^{}_{}}\\
\vdots&\vdots&\ddots&\vdots\\
B^{{}}_{n1}&B^{{}}_{n2}&\cdots&B^{{}}_{n^{}_{}n^{}_{}}\\
\end{pmatrix}^{}_{n^{}_{}\times n^{}_{}}
\begin{matrix}
L^{2}(\mu)\\
L^{2}(\mu)\\
\vdots\\
L^{2}(\mu)\\
\end{matrix}. \eqno{(3.18)}$$
For the $(n,2)$ entry of $AB$, by $AB=BA$ we obtain the equality
$$A^{{}}_{nn}B^{{}}_{n2}=B^{{}}_{n1}A^{{}}_{12}+B^{{}}_{n2}A^{{}}_{22}.
\eqno{(3.19)}$$ By $A^{{}}_{22}=A^{{}}_{nn}$ the equality $(3.19)$
yields that $B^{{}}_{n1}A^{{}}_{12}=0$. Since $A^{{}}_{12}$ is
supported on $\Lambda$, we obtain that $B^{{}}_{n1}=0$ a.e.~$[\mu]$
on $\Lambda$. Thus for the $(n,3)$ entry of $AB$, we obtain the
equality
$$A^{{}}_{nn}B^{{}}_{n3}=B^{{}}_{n2}A^{{}}_{23}+B^{{}}_{n3}A^{{}}_{33}.
\eqno{(3.20)}$$ By $A^{{}}_{33}=A^{{}}_{nn}$ the equality $(3.20)$
yields that $B^{{}}_{n2}A^{{}}_{23}=0$. Since $A^{{}}_{23}$ is
supported on $\Lambda$, we obtain that $B^{{}}_{n2}=0$ a.e.~$[\mu]$
on $\Lambda$. In this way and by the equalities with respect to the
entries in the $n$-th row of $AB$, we obtain that $B^{{}}_{ni}=0$
a.e.~$[\mu]$ on $\Lambda$ for $i=1,\ldots,n-1$. Based on this, by
the equalities with respect to the entries in the $(n-1)$-th row of
$AB$ and the fact that $A^{{}}_{i,i+1}$ is supported on $\Lambda$
for $i=1,\ldots,n-1$, we obtain that $B^{{}}_{n-1,i}=0$ a.e.~$[\mu]$
on $\Lambda$ for $i=1,\ldots,n-2$. Therefore, we finally obtain that
$B^{}_{ij}=0$ for $i>j$.

For the $(i,i+1)$ entry of $AB$, by $AB=BA$ we obtain the equality
$$A^{{}}_{ii}B^{{}}_{i,i+1}+A^{{}}_{i,i+1}B^{{}}_{i+1,i+1}=
B^{{}}_{ii}A^{{}}_{i,i+1}+B^{{}}_{i,i+1}A^{{}}_{i+1,i+1}.
\eqno{(3.21)}$$ By $A^{{}}_{ii}=A^{{}}_{i+1,i+1}$ the equality
$(3.21)$ yields that $A^{{}}_{i,i+1}B^{{}}_{i+1,i+1}=
B^{{}}_{ii}A^{{}}_{i,i+1}$. Since $A^{{}}_{i,i+1}$ is supported on
$\Lambda$, we obtain that $B^{{}}_{ii}=B^{{}}_{i+1,i+1}$
a.e.~$[\mu]$ on $\Lambda$.
\end{proof}

By Lemma $3.2$, we obtain that every idempotent in
$\{A\}^{\prime}\cap M^{}_{n^{}_{}}(L^{\infty}(\mu))$ is diagonal. In
the following lemma, we deal with a general case, but there is some
loss in the result compared with the result in the preceding lemma.

\begin{lemma}
Let $A$ in $M^{}_{m^{}_{}}(L^{\infty}(\mu))$ and $B$ in
$M^{}_{n^{}_{}}(L^{\infty}(\mu))$ be of the upper triangular forms
$$A^{}_{{}}=\begin{pmatrix}
A^{{}}_{11}&A^{{}}_{12}&\cdots&A^{{}}_{1m^{}_{}}\\
0&A^{{}}_{22}&\cdots&A^{{}}_{2m^{}_{}}\\
\vdots&\vdots&\ddots&\vdots\\
0&0&\cdots&A^{{}}_{m^{}_{}m^{}_{}}\\
\end{pmatrix}^{}_{m^{}_{}\times m^{}_{}}
\begin{matrix}
L^{2}(\mu)\\
L^{2}(\mu)\\
\vdots\\
L^{2}(\mu)\\
\end{matrix}, \eqno{(3.22)}$$
$$B^{}_{{}}=\begin{pmatrix}
B^{{}}_{11}&B^{{}}_{12}&\cdots&B^{{}}_{1n^{}_{}}\\
0&B^{{}}_{22}&\cdots&B^{{}}_{2n^{}_{}}\\
\vdots&\vdots&\ddots&\vdots\\
0&0&\cdots&B^{{}}_{n^{}_{}n^{}_{}}\\
\end{pmatrix}^{}_{n^{}_{}\times n^{}_{}}
\begin{matrix}
L^{2}(\mu)\\
L^{2}(\mu)\\
\vdots\\
L^{2}(\mu)\\
\end{matrix} \eqno{(3.23)}$$ such that
\begin{enumerate}
\item $A^{{}}_{ii}=B^{{}}_{jj}$ for $i=1,2,\ldots,m$ and $j=1,2,\ldots,n$;
\item $A^{{}}_{i,i+1}$ and $B^{{}}_{j,j+1}$ are supported on
$\Lambda$ for $i=1,2,\ldots,m-1$ and $j=1,2,\ldots,n-1$.
\end{enumerate}
If $m\geq n$, and the equalities $AC=CB$ and $BD=DA$ hold for
operator-valued matrices $C=(C^{}_{ij})^{}_{1\leq i\leq m; 1\leq
j\leq n}$ in $M^{}_{m\times n}(L^{\infty}(\mu))$ and
$D=(D^{}_{ij})^{}_{1\leq i\leq n; 1\leq j\leq m}$ in $M^{}_{n\times
m}(L^{\infty}(\mu))$, then $C^{}_{ij}=0$ for $i>j$ and $D^{}_{ij}=0$
for $i>j-(m-n)$.
\end{lemma}

\begin{proof}
By $AC=CB$, we consider the $(m-1,1)$ entry of $AC$ and obtain that
$$A^{}_{m-1,m-1}C^{}_{m-1,1}+A^{}_{m-1,m}C^{}_{m1}=C^{}_{m-1,1}B^{}_{11}.
\eqno{(3.24)}$$ By $A^{}_{m-1,m-1}=B^{}_{11}$, the equality $(3.24)$
yields that $A^{}_{m-1,m}C^{}_{m1}=0$. Since $A^{}_{m-1,m}$ is
supported on $\Lambda$, we obtain that $C^{}_{m1}=0$ a.e.~$[\mu]$ on
$\Lambda$. Thus we consider the $(m-2,1)$ entry of $AC$ and obtain
the equality
$$A^{}_{m-2,m-2}C^{}_{m-2,1}+A^{}_{m-2,m-1}C^{}_{m-1,1}=C^{}_{m-2,1}B^{}_{11}.
\eqno{(3.25)}$$  By $A^{}_{m-2,m-2}=B^{}_{11}$ and the fact
$A^{}_{m-2,m-1}$ is supported on $\Lambda$, the equality $(3.25)$
yields that $C^{}_{m-1,1}=0$ a.e.~$[\mu]$ on $\Lambda$. By the
equalities with respect to the entries in the first column of $AC$,
we obtain that $C^{}_{i1}=0$ for $i=m,m-1,\ldots,2$. Based on this,
the equalities with respect to the entries in the second column of
$AC$ yield that $C^{}_{i2}=0$ for $i=m,m-1,\ldots,3$. In this way,
we finally obtain that $C^{}_{ij}=0$ for $i>j$.

By $BD=DA$, we consider the $(n,2)$ entry of $BD$ and obtain that
$$B^{}_{nn}D^{}_{n2}=D^{}_{n1}A^{}_{12}+D^{}_{n2}A^{}_{22}.
\eqno{(3.26)}$$ By $A^{}_{22}=B^{}_{nn}$, the equality $(3.26)$
yields that $D^{}_{n1}A^{}_{12}=0$. Since $A^{}_{12}$ is supported
on $\Lambda$, we obtain that $D^{}_{n1}=0$ a.e.~$[\mu]$ on
$\Lambda$. Thus we consider the $(n,3)$ entry of $BD$ and obtain the
equality
$$B^{}_{nn}D^{}_{n3}=D^{}_{n2}A^{}_{23}+D^{}_{n3}A^{}_{33}.
\eqno{(3.27)}$$ By $A^{}_{33}=B^{}_{nn}$, the equality $(3.27)$
yields that $D^{}_{n2}A^{}_{23}=0$. Since $A^{}_{23}$ is supported
on $\Lambda$, we obtain that $D^{}_{n2}=0$ a.e.~$[\mu]$ on
$\Lambda$. By the equalities with respect to the entries in the
$n$-th row of $BD$, we obtain that $D^{}_{ni}=0$ for
$i=1,\ldots,m-1$. Based on this, the equalities with respect to the
entries in the $(n-1)$-th row of $BD$ yield that $D^{}_{n-1,i}=0$
for $i=1,\ldots,m-2$. In this way, we finally obtain that
$D^{}_{ij}=0$ for $i+(m-n)>j$.
\end{proof}

From Lemma $3.4$ to Lemma $3.8$, we make preparations for proving
Theorem $3.9$. For each operator $A$ in
$M^{}_{n^{}_{}}(L^{\infty}(\mu))$ with the relative commutant
$\{A\}^{\prime}\cap M^{}_{n^{}_{}}(L^{\infty}(\mu))$ containing a
finite frame, we can decompose $A$ into two fundamental cases by
Proposition $2.8$ and the discussion preceding it. In the following
paragraphs, Lemma $3.4$ and Proposition $3.5$ deal with one case
while Lemma $3.6$ and Proposition $3.7$ deal with another case.
Lemma $3.8$ shows the reason why we make the decomposition.

\begin{lemma}
For $k=1,\ldots,m$, let $A^{}_{k}=(A^{k}_{ij})^{}_{1\leq i,j\leq n}$
be an upper triangular operator-valued matrix in
$M^{}_{n^{}_{}}(L^{\infty}(\mu))$ such that
\begin{enumerate}
\item $A^{k}_{ii}=A^{l}_{jj}$ for $k,l=1,2,\ldots,m$ and $i,j=1,2,\ldots,n$;
\item $A^{k}_{i,i+1}$ is supported on
$\Lambda$ for $k=1,2,\ldots,m$ and $i=1,2,\ldots,n-1$.
\end{enumerate}
If $P$ is an idempotent in $\{\sum^{m}_{k=1}\oplus
A^{}_{k}\}^{\prime}\cap M^{}_{mn^{}_{}}(L^{\infty}(\mu))$, then
there exists an invertible operator $X$ in $\{\sum^{m}_{k=1}\oplus
A^{}_{k}\}^{\prime}\cap M^{}_{mn^{}_{}}(L^{\infty}(\mu))$ such that
$XPX^{-1}$ is diagonal.
\end{lemma}

\begin{proof}
Let $P$ be an idempotent in $\{\sum^{m}_{k=1}\oplus
A^{}_{k}\}^{\prime}\cap M^{}_{mn^{}_{}}(L^{\infty}(\mu))$ of the
form
$$P=\begin{pmatrix}
P^{}_{11}&P^{}_{12}&\cdots&P^{}_{1m}\\
P^{}_{21}&P^{}_{22}&\cdots&P^{}_{2m}\\
\vdots&\vdots&\ddots&\vdots\\
P^{}_{m1}&P^{}_{m2}&\cdots&P^{}_{mm}\\
\end{pmatrix}, \eqno{(3.28)}$$ where $P^{}_{kl}\in
M^{}_{n^{}_{}}(L^{\infty}(\mu))$ and
$A^{}_{k}P^{}_{kl}=P^{}_{kl}A^{}_{l}$ for $k,l=1,\ldots,m$. By Lemma
$3.2$ and Lemma $3.3$, we obtain that $P^{}_{kl}$ is of the upper
triangular form
$$P^{}_{kl}=\begin{pmatrix}
P^{kl}_{11}&P^{kl}_{12}&\cdots&P^{kl}_{1n}\\
0&P^{kl}_{22}&\cdots&P^{kl}_{2n}\\
\vdots&\vdots&\ddots&\vdots\\
0&0&\cdots&P^{kl}_{nn}\\
\end{pmatrix}^{}_{n\times n}
\begin{matrix}
L^{2}(\mu)\\
L^{2}(\mu)\\
\vdots\\
L^{2}(\mu)\\
\end{matrix}, \eqno{(3.29)}$$ where we treat $P^{kl}_{ij}$ as an
element in $L^{\infty}(\mu)$ for $k,l=1,\ldots,m$ and
$i,j=1,\ldots,n$. The following facts are important:
\begin{enumerate}
\item the equality $P^{kk}_{ii}=P^{kk}_{11}$ holds for
$k=1,\ldots,m$ and $i=1,\dots,n$;
\item the equality
$$A^{k}_{i,i+1}P^{kl}_{i+1,i+1}=P^{kl}_{ii}A^{l}_{i,i+1}
\eqno{(3.30)}$$ holds for $k,l=1,\ldots,m$ and $i=1,\ldots,n-1$.
Thus by the assumption that $A^{k}_{i,i+1}$ is supported on
$\Lambda$ for $k=1,\ldots,m$ and $i=1,\ldots,n-1$, the equality
$(3.30)$ yields that if $P^{kl}_{ii}=0$ for some $i$ then each main
diagonal entry of $P^{}_{kl}$ is $0$.
\end{enumerate}

There exists a unitary operator $U$ in
$M^{}_{mn^{}_{}}(L^{\infty}(\mu))$ which is a composition of
finitely many row-switching transformations such that $UPU^{*}(=Q)$
as an element in $\{U(\sum^{m}_{k=1}\oplus
A^{}_{k})U^{*}\}^{\prime}\cap M^{}_{mn^{}_{}}(L^{\infty}(\mu))$ is
of the upper triangular form
$$Q=\begin{pmatrix}
Q^{}_{11}&Q^{}_{12}&\cdots&Q^{}_{1n}\\
0&Q^{}_{22}&\cdots&Q^{}_{2n}\\
\vdots&\vdots&\ddots&\vdots\\
0&0&\cdots&Q^{}_{nn}\\
\end{pmatrix}, \eqno{(3.31)}$$ where for $i,j=1,\ldots,n$, every
$Q^{}_{ij}$ is of the form
$$Q^{}_{ij}=\begin{pmatrix}
P^{11}_{ij}&P^{12}_{ij}&\cdots&P^{1m}_{ij}\\
P^{21}_{ij}&P^{22}_{ij}&\cdots&P^{2m}_{ij}\\
\vdots&\vdots&\ddots&\vdots\\
P^{m1}_{ij}&P^{m2}_{ij}&\cdots&P^{mm}_{ij}\\
\end{pmatrix}^{}_{m\times m}
\begin{matrix}
L^{2}(\mu)\\
L^{2}(\mu)\\
\vdots\\
L^{2}(\mu)\\
\end{matrix}. \eqno{(3.32)}$$ Note that $Q^{}_{ii}$ is an idempotent
in $M^{}_{m^{}_{}}(L^{\infty}(\mu))$ for $i=1,\ldots,n$.

We treat Borel representatives as elements in $L^{\infty}(\mu)$.
Without loss of generality, we assume that $\mathrm{Tr}(Q^{}_{11})$
defined as in $(2.25)$ takes a constant $r>0$ on $\Lambda$. By the
proof of Lemma $2.3$, we may further assume that
$|P^{11}_{11}(\lambda)|\geq rn^{-1}$ for every $\lambda$ in
$\Lambda$. Therefore we construct an invertible element $X^{}_{ii}$
in $M^{}_{m^{}_{}}(L^{\infty}(\mu))$ for $i=1,\ldots,n$ of the form
$$X^{}_{ii}=\begin{pmatrix}
\mathbf{1}&0&\cdots&0\\
P^{21}_{ii}\slash P^{11}_{ii}&\mathbf{1}&\cdots&0\\
\vdots&\vdots&\ddots&\vdots\\
P^{m1}_{ii}\slash P^{11}_{ii}&0&\cdots&\mathbf{1}\\
\end{pmatrix}^{}_{m\times m}
\begin{matrix}
L^{2}(\mu)\\
L^{2}(\mu)\\
\vdots\\
L^{2}(\mu)\\
\end{matrix}. \eqno{(3.33)}$$
Therefore by the proof of Lemma $2.3$, the $(1,1)$ entry of
$X^{}_{ii}Q^{}_{ii}X^{-1}_{ii}$ is $\mathbf{1}$ and the $(k,1)$
entries of $X^{}_{ii}Q^{}_{ii}X^{-1}_{ii}$ are $0$s for
$k=2,\ldots,m$.

Based on the preceding discussion, let $Y$ in
$M^{}_{mn^{}_{}}(L^{\infty}(\mu))$ be of the form
$$Y=\begin{pmatrix}
Y^{}_{11}&Y^{}_{12}&\cdots&Y^{}_{1m}\\
Y^{}_{21}&Y^{}_{22}&\cdots&Y^{}_{2m}\\
\vdots&\vdots&\ddots&\vdots\\
Y^{}_{m1}&Y^{}_{m2}&\cdots&Y^{}_{mm}\\
\end{pmatrix}, \eqno{(3.34)}$$ where $Y^{}_{kl}\in
M^{}_{n^{}_{}}(L^{\infty}(\mu))$ and
$A^{}_{k}Y^{}_{kl}=Y^{}_{kl}A^{}_{l}$ for $k,l=1,\ldots,m$ such
that:
\begin{enumerate}
\item $Y^{}_{kk}$ equals the identity of
$M^{}_{n^{}_{}}(L^{\infty}(\mu))$ for $k=1,\ldots,m$;
\item  $Y^{k1}_{ij}=P^{k1}_{ij}\slash P^{11}_{11}$
for every entry $Y^{k1}_{ij}$ of $Y^{}_{k1}$ and $k=2,\ldots,m$;
\item other $Y^{}_{kl}$s not mentioned in the first two items are
$0$s for $k,l=1,\ldots,m$.
\end{enumerate} Note that every $Y^{}_{kl}\in
M^{}_{n^{}_{}}(L^{\infty}(\mu))$ is of the upper triangular form for
$k,l=1,\ldots,m$. Then $Y$ is an invertible element in
$\{\sum^{m}_{k=1}\oplus A^{}_{k}\}^{\prime}\cap
M^{}_{mn^{}_{}}(L^{\infty}(\mu))$. We write $X^{}_{1}=UYU^{*}$ for
$U$ in the discussion preceding $(3.31)$. Then $X^{}_{1}$ is of the
upper triangular form
$$X^{}_{1}=\begin{pmatrix}
X^{}_{11}&X^{}_{12}&\cdots&X^{}_{1n}\\
0&X^{}_{22}&\cdots&X^{}_{2n}\\
\vdots&\vdots&\ddots&\vdots\\
0&0&\cdots&X^{}_{nn}\\
\end{pmatrix}, \eqno{(3.35)}$$ where $X^{}_{ii}$ is in the form of
$(3.33)$ for $i=1,\ldots,n$. Therefore, $X^{}_{1}$ is invertible in
$\{U(\sum^{m}_{k=1}\oplus A^{}_{k})U^{*}\}^{\prime}\cap
M^{}_{mn^{}_{}}(L^{\infty}(\mu))$. In $X^{}_{1}QX^{-1}_{1}$, the
$(i,i)$ block entry $X^{}_{ii}Q^{}_{ii}X^{-1}_{ii}$ possesses the
property mentioned in the discussion preceding $(3.34)$. By a
similar proof of Lemma $2.3$ and iterating the preceding discussion,
we obtain an invertible element $X$ in $\{U(\sum^{m}_{k=1}\oplus
A^{}_{k})U^{*}\}^{\prime}\cap M^{}_{mn^{}_{}}(L^{\infty}(\mu))$ such
that $XQX^{-1}$ is of the upper triangular form
$$XQX^{-1}=\begin{pmatrix}
R^{}_{11}&R^{}_{12}&\cdots&R^{}_{1n}\\
0&R^{}_{22}&\cdots&R^{}_{2n}\\
\vdots&\vdots&\ddots&\vdots\\
0&0&\cdots&R^{}_{nn}\\
\end{pmatrix} \eqno{(3.36)}$$ where $R^{}_{ii}$ is diagonal for
$i=1,\ldots,n$. Thus $R^{}_{ii}=R^{}_{jj}$ for $i,j=1,\ldots,n$.
Note that the diagonal matrix
$$R^{}_{1}=\begin{pmatrix}
R^{}_{11}&0&\cdots&0\\
0&R^{}_{22}&\cdots&0\\
\vdots&\vdots&\ddots&\vdots\\
0&0&\cdots&R^{}_{nn}\\
\end{pmatrix} \eqno{(3.37)}$$ is in $\{U(\sum^{m}_{k=1}\oplus
A^{}_{k})U^{*}\}^{\prime}\cap M^{}_{mn^{}_{}}(L^{\infty}(\mu))$.
Therefore $XQX^{-1}-R^{}_{1}(=R)$ is in $\{U(\sum^{m}_{k=1}\oplus
A^{}_{k})U^{*}\}^{\prime}\cap M^{}_{mn^{}_{}}(L^{\infty}(\mu))$. For
every element $S$ in $\mathfrak{A}$, we obtain that
$\sigma^{}_{\mathfrak{A}}(SR)=\sigma^{}_{\mathfrak{A}}(RS)=\{0\}$,
where $\mathfrak{A}=\{U(\sum^{m}_{k=1}\oplus
A^{}_{k})U^{*}\}^{\prime}\cap M^{}_{mn^{}_{}}(L^{\infty}(\mu))$.
Therefore the equality
$(I-2XQX^{-1})(I-2XQX^{-1}+R)=I+(I-2XQX^{-1})R$ yields that
$I-2XQX^{-1}+R$ is invertible in $\{U(\sum^{m}_{k=1}\oplus
A^{}_{k})U^{*}\}^{\prime}\cap M^{}_{mn^{}_{}}(L^{\infty}(\mu))$,
since $(I-2XQX^{-1})^{2}=I$. Hence the equality
$$\begin{array}{ccl}
(XQX^{-1}-R)^{2}&=&XQX^{-1}+R^{2}-RXQX^{-1}-XQX^{-1}R\\
&=&XQX^{-1}-R\\
\end{array}  \eqno{(3.38)}$$ yields that
$R^{2}+R-RXQX^{-1}-XQX^{-1}R=0$, and we obtain
$$\begin{array}{rcl}
&&(XQX^{-1}-R)(I-2XQX^{-1}+R)\\
&=&2RXQX^{-1}-XQX^{-1}+XQX^{-1}R-R-R^{2}\\
&=&RXQX^{-1}-XQX^{-1}\\
&=&(R+I-2XQX^{-1})XQX^{-1}.\\
\end{array} \eqno{(3.39)}$$ This means that $XQX^{-1}$ is similar to
$XQX^{-1}-R(=R^{}_{1})$ in the relative commutant
$\{U(\sum^{m}_{k=1}\oplus A^{}_{k})U^{*}\}^{\prime}\cap
M^{}_{mn^{}_{}}(L^{\infty}(\mu))$. Thus $Z=U^{*}(R+I-2XQX^{-1})XU$
is the required invertible element in $\{\sum^{m}_{k=1}\oplus
A^{}_{k}\}^{\prime}\cap M^{}_{mn^{}_{}}(L^{\infty}(\mu))$ such that
$ZPZ^{-1}$ is diagonal.
\end{proof}

By the preceding lemma and the proof of Lemma $2.6$ we obtain the
following proposition.

\begin{proposition}
For $k=1,\ldots,m$, let $A^{}_{k}=(A^{k}_{ij})^{}_{1\leq i,j\leq n}$
be an upper triangular operator-valued matrix in
$M^{}_{n^{}_{}}(L^{\infty}(\mu))$ such that
\begin{enumerate}
\item $A^{k}_{ii}=A^{l}_{jj}$ for $k,l=1,\ldots,m$ and
$i,j=1,\ldots,n$;
\item $A^{k}_{i,i+1}$ is supported on
$\Lambda$ for $k=1,\ldots,m$ and $i=1,\ldots,n-1$.
\end{enumerate}
If $\mathscr{P}$ and $\mathscr{Q}$ are two bounded maximal abelian
sets of idempotents in the relative commutant
$\{\sum^{m}_{k=1}\oplus A^{}_{k}\}^{\prime}\cap
M^{}_{mn^{}_{}}(L^{\infty}(\mu))$, then there exists an invertible
operator $X$ in $\{\sum^{m}_{k=1}\oplus A^{}_{k}\}^{\prime}\cap
M^{}_{mn^{}_{}}(L^{\infty}(\mu))$ such that
$X\mathscr{P}X^{-1}=\mathscr{Q}$.
\end{proposition}

\begin{proof}
Denote by $\mathscr{D}^{}_{}$ the set of diagonal projections in
$\{\sum^{m}_{k=1}\oplus A^{}_{k}\}^{\prime}\cap
M^{}_{mn^{}_{}}(L^{\infty}(\mu))$. We can verify that
$\mathscr{D}^{}_{}$ is a bounded maximal abelian set of idempotents
in the relative commutant $\{\sum^{m}_{k=1}\oplus
A^{}_{k}\}^{\prime}\cap M^{}_{mn^{}_{}}(L^{\infty}(\mu))$. By
Proposition $2.5$, there exists a finite frame $\mathscr{P}^{}_{0}$
in $\mathscr{P}^{}_{}$. Thus combining Lemma $3.4$ and the proof of
Lemma $2.6$, we obtain an invertible operator $X$ in
$\{\sum^{m}_{k=1}\oplus A^{}_{k}\}^{\prime}\cap
M^{}_{mn^{}_{}}(L^{\infty}(\mu))$ such that every element in
$X\mathscr{P}^{}_{0}X^{-1}$ is diagonal. Furthermore, every element
in $X\mathscr{P}^{}_{}X^{-1}$ is diagonal. Therefore
$X\mathscr{P}^{}_{}X^{-1}\subseteq\mathscr{D}$. The maximality of
$\mathscr{P}^{}_{}$ yields the equality
$X\mathscr{P}^{}_{}X^{-1}=\mathscr{D}$. By the same way, there
exists an invertible operator $Y$ in $\{\sum^{m}_{k=1}\oplus
A^{}_{k}\}^{\prime}\cap M^{}_{mn^{}_{}}(L^{\infty}(\mu))$ such that
$Y\mathscr{Q}^{}_{}Y^{-1}=\mathscr{D}$. Therefore $\mathscr{P}$ and
$\mathscr{Q}$ are similar to each other in $\{\sum^{m}_{k=1}\oplus
A^{}_{k}\}^{\prime}\cap M^{}_{mn^{}_{}}(L^{\infty}(\mu))$.
\end{proof}

In the following lemma, we pay attention to another case different
from the one mentioned in Lemma $3.4$.

\begin{lemma}
For $k=1,\ldots,m$, let $A^{}_{k}=(A^{k}_{ij})^{}_{1\leq i,j\leq
n^{}_{k}}$ be an upper triangular operator-valued matrix in
$M^{}_{n^{}_{k}}(L^{\infty}(\mu))$ such that
\begin{enumerate}
\item $A^{k}_{ii}=A^{l}_{jj}$ for $k,l=1,\ldots,m$,
$i=1,\ldots,n^{}_{k}$ and $j=1,\ldots,n^{}_{l}$;
\item $A^{k}_{i,i+1}$ is supported on
$\Lambda$ for $k=1,\ldots,m$ and $i=1,\ldots,n^{}_{k}-1$;
\item $n^{}_{1}>\cdots>n^{}_{m}$ and
$\sum^{m}_{k=1}n^{}_{k}=r$.
\end{enumerate}
If $P$ is an idempotent in $\{\sum^{m}_{k=1}\oplus
A^{}_{k}\}^{\prime}\cap M^{}_{r^{}_{}}(L^{\infty}(\mu))$, then there
exists an invertible operator $X$ in $\{\sum^{m}_{k=1}\oplus
A^{}_{k}\}^{\prime}\cap M^{}_{r^{}_{}}(L^{\infty}(\mu))$ such that
$XPX^{-1}$ is diagonal.
\end{lemma}

\begin{proof}
For the sake of simplicity, we only prove this lemma for $m=2$ and
$A^{}_{1}$, $A^{}_{2}$ are of the forms
$$A^{}_{1}=\begin{pmatrix}
A^{1}_{11}&A^{1}_{12}&A^{1}_{13}\\
0&A^{1}_{22}&A^{1}_{23}\\
0&0&A^{1}_{33}\\
\end{pmatrix}
\begin{matrix}
L^{2}(\mu)\\
L^{2}(\mu)\\
L^{2}(\mu)\\
\end{matrix},\quad
A^{}_{2}=\begin{pmatrix}
A^{2}_{11}&A^{2}_{12}\\
0&A^{2}_{22}\\
\end{pmatrix}
\begin{matrix}
L^{2}(\mu)\\
L^{2}(\mu)\\
\end{matrix}. \eqno{(3.40)}$$
If $P$ is an idempotent in the relative commutant $\{A^{}_{1}\oplus
A^{}_{2}\}^{\prime}\cap M^{}_{5^{}_{}}(L^{\infty}(\mu))$, then by
Lemma $3.2$ and Lemma $3.3$, $P$ is of the form
$$P=\left(\begin{array}{ccc|cc}
P^{11}_{11}&P^{11}_{12}&P^{11}_{13}&P^{12}_{11}&P^{12}_{12}\\
0&P^{11}_{22}&P^{11}_{23}&0&P^{12}_{22}\\
0&0&P^{11}_{33}&0&0\\
\hline
0&P^{21}_{12}&P^{21}_{13}&P^{22}_{11}&P^{22}_{12}\\
0&0&P^{21}_{23}&0&P^{22}_{22}\\
\end{array}\right)
\begin{matrix}
L^{2}(\mu)\\
L^{2}(\mu)\\
L^{2}(\mu)\\
L^{2}(\mu)\\
L^{2}(\mu)\\
\end{matrix}. \eqno{(3.41)}$$ There exists a unitary element $U$ in
$M^{}_{5^{}_{}}(L^{\infty}(\mu))$ which is a composition of finitely
many row-switching transformations such that $UPU^{*}(=Q)$ as an
element in the relative commutant $\mathfrak{A}=\{U(A^{}_{1}\oplus
A^{}_{2})U^{*}\}^{\prime}\cap M^{}_{5^{}_{}}(L^{\infty}(\mu))$ is of
the form
$$Q=\left(\begin{array}{cc|cc|c}
P^{11}_{11}&P^{12}_{11}&P^{11}_{12}&P^{12}_{12}&P^{11}_{13}\\
0&P^{22}_{11}&P^{21}_{12}&P^{22}_{12}&P^{21}_{13}\\
\hline
0&0&P^{11}_{22}&P^{12}_{22}&P^{11}_{23}\\
0&0&0&P^{22}_{22}&P^{21}_{23}\\
\hline
0&0&0&0&P^{11}_{33}\\
\end{array}\right)
\begin{matrix}
L^{2}(\mu)\\
L^{2}(\mu)\\
L^{2}(\mu)\\
L^{2}(\mu)\\
L^{2}(\mu)\\
\end{matrix}. \eqno{(3.42)}$$ We can verify that
$Q^{}_{1}=\mathrm{diag}
(P^{11}_{11},P^{22}_{11},P^{11}_{22},P^{22}_{22},P^{11}_{33})$ is a
projection in the relative commutant $\mathfrak{A}$. Thus
$R=Q-Q^{}_{1}$ is in the relative commutant $\mathfrak{A}$. By
computation, we obtain that the equality
$\sigma^{}_{\mathfrak{A}}(SR)=\sigma^{}_{\mathfrak{A}}(RS)=\{0\}$
holds for every element $S$ in $\mathfrak{A}$. Therefore by an
analogous proof following $(3.37)$ in Lemma $3.4$, there exists an
invertible element $X$ in the relative commutant $\{A^{}_{1}\oplus
A^{}_{2}\}^{\prime}\cap M^{}_{5^{}_{}}(L^{\infty}(\mu))$ such that
$XPX^{-1}$ is diagonal. By iterating the preceding proof we achieve
a generalized case mentioned in the lemma.
\end{proof}

By combining Proposition $2.5$, Lemma $3.6$ and the method applied
in the proof of Lemma $2.6$, we obtain the following proposition.

\begin{proposition}
For $k=1,\ldots,m$, let $A^{}_{k}=(A^{k}_{ij})^{}_{1\leq i,j\leq
n^{}_{k}}$ be an upper triangular operator-valued matrix in
$M^{}_{n^{}_{k}}(L^{\infty}(\mu))$ such that
\begin{enumerate}
\item $A^{k}_{ii}=A^{l}_{jj}$ for $k,l=1,\ldots,m$,
$i=1,\ldots,n^{}_{k}$ and $j=1,\ldots,n^{}_{l}$;
\item $A^{k}_{i,i+1}$ is supported on
$\Lambda$ for $k=1,\ldots,m$ and $i=1,\ldots,n^{}_{k}-1$;
\item $n^{}_{1}>\cdots>n^{}_{m}$ and
$\sum^{m}_{k=1}n^{}_{k}=r$.
\end{enumerate}
If $\mathscr{P}$ and $\mathscr{Q}$ are two bounded maximal abelian
sets of idempotents in the relative commutant
$\{\sum^{m}_{k=1}\oplus A^{}_{k}\}^{\prime}\cap
M^{}_{r^{}_{}}(L^{\infty}(\mu))$, then there exists an invertible
operator $X$ in $\{\sum^{m}_{k=1}\oplus A^{}_{k}\}^{\prime}\cap
M^{}_{r^{}_{}}(L^{\infty}(\mu))$ such that
$X\mathscr{P}X^{-1}=\mathscr{Q}$.
\end{proposition}

By a similar computation as in Lemma $3.3$, we obtain the following
lemma. This lemma implies the reason for our consideration about the
preceding two cases.

\begin{lemma}
For $k=1,2$, let $A^{}_{k}=(A^{k}_{ij})^{}_{1\leq i,j\leq n^{}_{k}}$
be an upper triangular operator-valued matrix in
$M^{}_{n^{}_{k}}(L^{\infty}(\mu))$ such that
\begin{enumerate}
\item $A^{k}_{ii}=A^{k}_{jj}$ for $k=1,2$, $i,j=1,\ldots,n^{}_{k}$;
\item $A^{1}_{11}(\lambda)\neq A^{2}_{11}(\lambda)$ a.e.~$[\mu]$ on
$\Lambda$.
\end{enumerate}
If $B$ is an element in $M^{}_{n^{}_{1}\times
n^{}_{2}}(L^{\infty}(\mu))$ such that $A^{}_{1}B=BA^{}_{2}$, then
$B=0$.
\end{lemma}

By an analogous discussion preceding Proposition $2.8$, for $A$ in
$M^{}_{n^{}_{}}(L^{\infty}(\mu))$ with the relative commutant
$\{A\}^{\prime}\cap M^{}_{n^{}_{}}(L^{\infty}(\mu))$ containing a
finite frame, there exists a finite $\mu$-measurable partition
$\{\Lambda^{}_{t}\}^{r}_{t=1}$ of $\Lambda$ such that for
$k,l=1,\ldots,m$ and based on the notations in Proposition $2.8$,
either $A^{k}_{11}(\lambda)=A^{l}_{11}(\lambda)$ a.e.~$[\mu]$ on
$\Lambda^{}_{t}$ or $A^{k}_{11}(\lambda)\neq A^{l}_{11}(\lambda)$
a.e.~$[\mu]$ on $\Lambda^{}_{t}$ for $t=1,\ldots,r$. Denote by
$E^{}_{t}$ the central projection in
$M^{}_{n^{}_{}}(L^{\infty}(\mu))$ such that $\mathrm{Tr}(E^{}_{t})$
is supported on $\Lambda^{}_{t}$ for $t=1,\ldots,r$. Then combining
Proposition $3.5$, Proposition $3.7$ and Lemma $3.8$, we obtain that
for every two bounded maximal abelian sets of idempotents
$\mathscr{P}$ and $\mathscr{Q}$ in the relative commutant
$\{AE^{}_{t}\}^{\prime}\cap
M^{}_{n^{}_{}}(L^{\infty}(\Lambda^{}_{t},\mu))$, there exists an
invertible element $X$ in $\{AE^{}_{t}\}^{\prime}\cap
M^{}_{n^{}_{}}(L^{\infty}(\Lambda^{}_{t},\mu))$ such that
$X\mathscr{P}X^{-1}=\mathscr{Q}$. Thus by combining the `local
results', we obtain the following theorem.

\begin{theorem}
Let $A$ be an operator in $M^{}_{n^{}_{}}(L^{\infty}(\mu))$. If the
the relative commutant $\{A\}^{\prime}\cap
M^{}_{n^{}_{}}(L^{\infty}(\mu))$ contains a finite frame then for
every two bounded maximal abelian sets of idempotents $\mathscr{P}$
and $\mathscr{Q}$ in the relative commutant $\{A\}^{\prime}\cap
M^{}_{n^{}_{}}(L^{\infty}(\mu))$, there exists an invertible element
$X$ in $\{A\}^{\prime}\cap M^{}_{n^{}_{}}(L^{\infty}(\mu))$ such
that $X\mathscr{P}X^{-1}=\mathscr{Q}$.
\end{theorem}

By virtue of Theorem $3.9$, we prove the following corollary.

\begin{corollary}
For every normal operator $A$ in $M^{}_{n^{}_{}}(L^{\infty}(\mu))$,
the strongly irreducible decomposition of $A$ is unique up to
similarity with respect to the relative commutant
$\{A\}^{\prime}\cap M^{}_{n^{}_{}}(L^{\infty}(\mu))$.
\end{corollary}

\begin{proof}
Let $A$ be a normal operator in $M^{}_{n^{}_{}}(L^{\infty}(\mu))$
and $\mathscr{P}$ be a bounded maximal abelian set of projections in
$\{A\}^{\prime}\cap M^{}_{n^{}_{}}(L^{\infty}(\mu))$.

First we assert that $\mathscr{P}$ is a bounded maximal abelian set
of idempotents in $\{A\}^{\prime}\cap
M^{}_{n^{}_{}}(L^{\infty}(\mu))$. Assume that $Q$ is an idempotent
in $\{A\}^{\prime}\cap M^{}_{n^{}_{}}(L^{\infty}(\mu))$ such that
the equality $QP=PQ$ holds for every projection in $\mathscr{P}$. By
the polar decomposition of $Q$ we can express $Q$ in the form
$Q=V|Q|$, where $V$ in $M^{}_{n^{}_{}}(L^{\infty}(\mu))$ is a
partial isometry with initial space $(\mathrm{ker}\, Q)^{\bot}$ and
finial space $\mathrm{ran}\, Q$. The equalities $PQ=QP$ and
$PQ^{*}=Q^{*}P$ yield that $P|Q|=|Q|P$. Therefore the definition of
$V$ and the equality
$$PV|Q|=PQ=QP=V|Q|P=VP|Q| \eqno{(3.43)}$$ yield that the equality
$PV=VP$ holds for every projection $P$ in $\mathscr{P}$. Denote by
$\mathfrak{A}$ the abelian von Neumann algebra generated by
$\mathscr{P}$. Since $\mathfrak{A}$ is a maximal abelian von Neumann
algebra in $\{A\}^{\prime}\cap M^{}_{n^{}_{}}(L^{\infty}(\mu))$, we
obtain that $A$ is contained in $\mathfrak{A}$. The facts that
$PV=VP$ holds for every projection $P$ in $\mathscr{P}$ and $A$ is
contained in $\mathfrak{A}$ yield that $AV=VA$ and $AV^{*}=V^{*}A$.
Therefore $V$ and $V^{*}$ are contained in $\{A\}^{\prime}\cap
M^{}_{n^{}_{}}(L^{\infty}(\mu))$. Thus $VV^{*}$ and $V^{*}V$ are
contained in $\mathfrak{A}$. Since the equality $QP=PQ$ holds for
every projection in $\mathscr{P}$, we obtain that $VV^{*}Q=QVV^{*}$.
By the definition of $V$, $VV^{*}$ is the final projection of $V$
with $\mathrm{ran}\, VV^{*}=\mathrm{ran}\, Q$. Hence the equality
$Q=VV^{*}Q=QVV^{*}=VV^{*}$ holds. The assertion is achieved. By
Proposition $2.5$ and Theorem $3.9$, we can finish the proof.
\end{proof}

\begin{corollary}
Let $A$ be a normal operator in $M^{}_{n^{}_{}}(L^{\infty}(\mu))$.
Then every bounded maximal abelian set of idempotents in
$\{A\}^{\prime}\cap M^{}_{n^{}_{}}(L^{\infty}(\mu))$ is a bounded
maximal abelian set of idempotents in $\{A\}^{\prime}$.
\end{corollary}

\begin{proof}
Note that every bounded maximal abelian set of idempotents in
$\{A\}^{\prime}\cap M^{}_{n^{}_{}}(L^{\infty}(\mu))$ contains the
set of central projections in $M^{}_{n^{}_{}}(L^{\infty}(\mu))$. Let
$\mathscr{P}$ be a bounded maximal abelian set of idempotents in
$\{A\}^{\prime}\cap M^{}_{n^{}_{}}(L^{\infty}(\mu))$ and $Q$ be an
idempotent in $\{A\}^{\prime}\cap\{\mathscr{P}\}^{\prime}$. Since
$Q$ commutes with every central projection in
$M^{}_{n^{}_{}}(L^{\infty}(\mu))$, we obtain that $Q$ is contained
in $M^{}_{n^{}_{}}(L^{\infty}(\mu))$. Thus $Q$ is contained in
$\mathscr{P}$. Therefore $\mathscr{P}$ is a bounded maximal abelian
set of idempotents in $\{A\}^{\prime}$.
\end{proof}

\section{Direct integral forms and the `local' $K$-theory for the
relative commutants of operators in
$M^{}_{n^{}_{}}(L^{\infty}(\mu))$}

By virtue of the reduction theory of von Neumann algebras, for every
operator $A$ in $M^{}_{n^{}_{}}(L^{\infty}(\mu))$, we can express
$A$ in a direct integral form with respect to an abelian von Neumann
algebra in the relative commutant $\{A\}^{\prime}\cap
M^{}_{n^{}_{}}(L^{\infty}(\mu))$. A natural question is what the
direct integral form of $A$ looks like with respect to a finite
frame in $\{A\}^{\prime}\cap M^{}_{n^{}_{}}(L^{\infty}(\mu))$. For
this purpose we need to introduce some concepts and notations that
will be used in this section. First, we need two concepts from
operator theory.

\begin{definition}
An operator $A$ in $\mathscr {L}(\mathscr {H})$ is said to be
{\textit{irreducible}} if its commutant
$\{A\}^{\prime}\triangleq\{B\in\mathscr {L}(\mathscr {H}):AB=BA\}$
contains no projections other than $0$ and the identity operator $I$
on $\mathscr{H}$, introduced by P. Halmos in \cite{Halmos}. (The
separability assumption of $\mathscr {H}$ is necessary because on a
nonseparable Hilbert space every operator is reducible.) An operator
$A$ in $\mathscr {L}(\mathscr {H})$ is said to be {\textit{strongly
irreducible}} if $XAX^{-1}$ is irreducible for every invertible
operator $X$ in $\mathscr {L}(\mathscr {H})$, introduced by F.
Gilfeather in \cite{Gilfeather}. This shows that the commutant of a
strongly irreducible operator contains no idempotents other than $0$
and $I$. For more literature around strongly irreducible operators,
the reader is referred to \cite{Jiang_5}.
\end{definition}

We also need some concepts from the reduction theory of von Neumann
algebras.

\begin{definition}
For the most part, we follow \cite{Azoff_2,Schwartz}. Once and for
all, Let $\mathscr {H}^{}_{1}\subset\mathscr
{H}^{}_{2}\subset\cdots\subset\mathscr {H}^{}_{\infty}$ be a
sequence of Hilbert spaces with $\mathscr {H}^{}_{n}$ having
dimension $n$ and $\mathscr {H}^{}_{\infty}$ spanned by the
remaining $\mathscr {H}^{}_n$s. Let $\mu$ be (the completion of) a
finite positive regular Borel measure supported on a compact subset
$\Lambda$ of $\mathbb {R}$. (This assumption makes sense by virtue
of (\cite{Rosenthal}, Theorem $7.12$).) And let
$\{\Lambda^{}_{\infty}\}\cup\{\Lambda^{}_{n}\}^{\infty}_{n=1}$ be a
Borel partition of $\Lambda$. Then we form the associated direct
integral Hilbert space $$\mathscr {H}=\int^{\oplus}_{\Lambda}
\mathscr {H}(\lambda) d\mu(\lambda), \eqno{(4.1)}$$ which consists
of all (equivalence classes of) measurable functions $f$ and $g$
from $\Lambda$ into $\mathscr {H}^{}_{\infty}$ such that:
\begin{enumerate}
\item[(1)] $f(\lambda)\in \mathscr {H}(\lambda)\equiv \mathscr
{H}^{}_{n}$ for $\lambda\in \Lambda^{}_{n}$;
\item[(2)] $\|f\|^{2}\triangleq\int^{}_{\Lambda}\|f(\lambda)\|^{2}
d\mu(\lambda)<\infty$;
\item[(3)] $(f,g)\triangleq\int^{}_{\Lambda}(f(\lambda),g(\lambda))
d\mu(\lambda)$.
\end{enumerate}
The element in $\mathscr {H}$ represented by the measurable function
$\lambda\rightarrow f(\lambda)$ is denoted by
$\int^{\oplus}_{\Lambda}f(\lambda) d\mu(\lambda)$. An operator $A$
in $\mathscr {L}(\mathscr {H})$ is said  to be
{\textit{decomposable}} if there exists a strongly $\mu$-measurable
operator-valued function $A(\cdot)$ defined on $\Lambda$ such that
$A(\lambda)$ is an operator in $\mathscr {L}(\mathscr {H}(\lambda))$
and $(Af)(\lambda)=A(\lambda)f(\lambda)$, for every $f\in \mathscr
{H}$. We write
$A\equiv\int^{\oplus}_{\Lambda}A(\lambda)d\mu(\lambda)$ for the
equivalence class corresponding to $A(\cdot)$.  If $A(\lambda)$ is a
scalar multiple of the identity on $\mathscr {H}(\lambda)$ for
almost every $\lambda$ in $\Lambda$, then $A$ is said to be
{\textit{diagonal}}. The collection of all diagonal operators is
said to be the {\textit{diagonal algebra}} of $\Lambda$. It is an
abelian von Neumann algebra. If $A(\lambda)$ is strongly irreducible
on $\mathscr {H}(\lambda)$ for almost every $\lambda$ in $\Lambda$,
then $A$ is said to be {\textit{a direct integral of strongly
irreducible operators}}.
\end{definition}

The following two basic results will be used in the sequel:
\begin{enumerate}
\item An operator acting on a direct integral of Hilbert spaces is
decomposable if and only if it commutes with the corresponding
diagonal algebra (\cite{Schwartz}, p. $22$).
\item Every abelian von Neumann algebra is (unitarily equivalent to)
an essentially unique diagonal algebra (\cite{Schwartz}, p. $19$).
\end{enumerate}

For a direct integral of strongly irreducible operators, we obtain
the following proposition.

\begin{proposition}
Let $A$ be in $M^{}_{n^{}_{}}(L^{\infty}(\mu))$. Then
$\{A\}^{\prime}\cap M^{}_{n^{}_{}}(L^{\infty}(\mu))$ contains a
finite frame if and only if there exists an invertible element $X$
in $M^{}_{n^{}_{}}(L^{\infty}(\mu))$ and a unitary operator $U$ such
that $UXAX^{-1}U^{*}_{}$ is a direct integral of strongly
irreducible operators with respect to a diagonal algebra
$\mathscr{D}$ and $U^{*}_{}\mathscr{D}U\subseteq
M^{}_{n^{}_{}}(L^{\infty}(\mu))$.
\end{proposition}

\begin{proof}
If $\{A\}^{\prime}\cap M^{}_{n^{}_{}}(L^{\infty}(\mu))$ contains a
finite frame $\{P^{}_{k}\}^{m}_{k=1}$, then by Lemma $2.6$ there
exists an invertible element $X$ in
$M^{}_{n^{}_{}}(L^{\infty}(\mu))$ such that every $XP^{}_{k}X^{-1}$
is diagonal for $k=1,2,\ldots,m$. The set
$\{XP^{}_{k}X^{-1}\}^{m}_{k=1}$ is a finite frame in the relative
commutant $\{XAX^{-1}\}^{\prime}\cap
M^{}_{n^{}_{}}(L^{\infty}(\mu))$. Therefore
$\{XP^{}_{k}X^{-1}\}^{m}_{k=1}$ and $\mathscr{E}^{}_{n}$, the set of
central projections of $M^{}_{n^{}_{}}(L^{\infty}(\mu))$, generate a
maximal abelian von Neumann subalgebra $\mathscr{D}^{}_{0}$ in the
relative commutant $\{XAX^{-1}\}^{\prime}\cap
M^{}_{n^{}_{}}(L^{\infty}(\mu))$. Since $\mathscr{D}^{}_{0}$
contains $\mathscr{E}^{}_{n}$, we can verify that
$\mathscr{D}^{}_{0}$ is also a maximal abelian von Neumann
subalgebra in $\{XAX^{-1}\}^{\prime}$. Thus by (\cite{Rosenthal},
Theorem $7.12$), there exists a self-adjoint element $D^{}_{0}$
generating $\mathscr{D}^{}_{0}$.

For the sake of simplicity, we assume that $A$ in
$M^{}_{3^{}_{}}(L^{\infty}(\mu))$ is of the form
$$A=\begin{pmatrix}
f^{}_{}&f^{}_{12}&f^{}_{13}\\
0&f^{}_{}&f^{}_{23}\\
0&0&f^{}_{}\\
\end{pmatrix}
\begin{matrix}
L^{2}(\mu)\\
L^{2}(\mu)\\
L^{2}(\mu)\\
\end{matrix} \eqno{(4.2)}$$ such that:
\begin{enumerate}
\item $\Lambda=\Lambda^{}_{1}\cup\Lambda^{}_{2}$ and
$\mu(\Lambda^{}_{i})>0$ for $i=1,2$;
\item the support of $f^{}_{12}$ is $\Lambda$;
\item the support of $f^{}_{23}$ is $\Lambda^{}_{1}$;
\item $f^{}_{13}$ and $f^{}_{23}$ vanish on $\Lambda^{}_{2}$.
\end{enumerate} Correspondingly, $\{P^{}_{1}, P^{}_{2}\}$ is a
finite self-adjoint frame contained in the relative commutant
$\{A\}^{\prime}\cap M^{}_{3^{}_{}}(L^{\infty}(\mu))$ such that
\begin{enumerate}
\item $P^{}_{1}$ is of the form
$$P^{}_{1}(\lambda)=\begin{cases}
\mathrm{diag}(1,1,1)\in
M^{}_{3}(\mathbb{C}),&\forall\lambda\in\Lambda^{}_{1};\\
\mathrm{diag}(1,1,0)\in
M^{}_{3}(\mathbb{C}),&\forall\lambda\in\Lambda^{}_{2};\\
\end{cases} \eqno{(4.3)}$$
\item $P^{}_{2}$ is of the form
$$P^{}_{2}(\lambda)=\begin{cases}
0\in M^{}_{3}(\mathbb{C}),&\forall\lambda\in\Lambda^{}_{1};\\
\mathrm{diag}(0,0,1)\in
M^{}_{3}(\mathbb{C}),&\forall\lambda\in\Lambda^{}_{2}.\\
\end{cases} \eqno{(4.4)}$$
\end{enumerate} Thus the abelian von Neumann algebra $\mathscr{D}^{}_{0}$
generated by $\{P^{}_{1}, P^{}_{2}\}\cup\mathscr{E}^{}_{3}$ is
maximal in the relative commutant $\{A\}^{\prime}\cap
M^{}_{3^{}_{}}(L^{\infty}(\mu))$. By (\cite{Rosenthal}, Theorem
$7.12$), there exist self-adjoint injective elements $D^{}_{i}$ in
either $L^{\infty}(\Lambda^{}_{1},\mu)$ or
$L^{\infty}(\Lambda^{}_{2},\mu)$ for $i=1,2,3$ such that:
\begin{enumerate}
\item $D^{}_{i}$ is in $L^{\infty}(\Lambda^{}_{i},\mu)$ with
$\{D^{}_{i}\}^{\prime\prime}_{}=L^{\infty}(\Lambda^{}_{i},\mu)$ for
$i=1,2$, and $D^{}_{3}$ is in $L^{\infty}(\Lambda^{}_{2},\mu)$ with
$\{D^{}_{3}\}^{\prime\prime}_{}=L^{\infty}(\Lambda^{}_{2},\mu)$;
\item the spectra of $D^{}_{i}$s denoted by $\Gamma^{}_{i}$ for $i=1,2,3$
are pairwise disjoint.
\end{enumerate}
Let $D^{}_{0}$ be a self-adjoint element in $\{A\}^{\prime}\cap
M^{}_{3^{}_{}}(L^{\infty}(\mu))$ of the form
$$D^{}_{0}(\lambda)=\begin{cases}
\mathrm{diag}(D^{}_{1}(\lambda),D^{}_{1}(\lambda),D^{}_{1}(\lambda))\in
M^{}_{3}(\mathbb{C}),&\forall\lambda\in\Lambda^{}_{1};\\
\mathrm{diag}(D^{}_{2}(\lambda),D^{}_{2}(\lambda),D^{}_{3}(\lambda))\in
M^{}_{3}(\mathbb{C}),&\forall\lambda\in\Lambda^{}_{2}.\\
\end{cases} \eqno{(4.5)}$$ Therefore
$\sigma(D^{}_{0})=\Gamma^{}_{1}\cup\Gamma^{}_{2}\cup\Gamma^{}_{3}$
and $\{D^{}_{0}\}^{\prime\prime}_{}=\mathscr{D}^{}_{0}$. Denote by
$E^{}_{D^{}_{0}}(\cdot)$ the spectral measure for $D^{}_{0}$. With
respect to $\sigma(D^{}_{0})$, we form the direct integral Hilbert
space $\mathscr{H}$ as follows
$$\mathscr {H}=\int^{\oplus}_{\sigma(D^{}_{0})}
\mathscr {H}(\gamma) d\nu(\gamma), \eqno{(4.6)}$$ where
$$\mathscr{H}(\gamma)=\begin{cases}
\mathbb{C}^{}_{}\oplus\mathbb{C}^{}_{}\oplus\mathbb{C}^{}_{},&\forall\gamma\in\Gamma^{}_{1};\\
\mathbb{C}^{}_{}\oplus\mathbb{C}^{}_{}\oplus 0,&\forall\gamma\in\Gamma^{}_{2};\\
0\oplus 0\oplus\mathbb{C}^{}_{},&\forall\gamma\in\Gamma^{}_{3};\\
\end{cases} \eqno{(4.7)}$$
$$\nu(\Delta)=\begin{cases}
\mu\circ D^{-1}_{1}(\Delta),&\forall\Delta\subseteq\Gamma^{}_{1};\\
\mu\circ D^{-1}_{2}(\Delta),&\forall\Delta\subseteq\Gamma^{}_{2};\\
\mu\circ D^{-1}_{3}(\Delta),&\forall\Delta\subseteq\Gamma^{}_{3};\\
\end{cases} \eqno{(4.8)}$$ and $D^{-1}_{i}(\Delta)$ denotes the
preimage of each Borel subset $\Delta$ with respect to (the Borel
function) $D^{}_{i}$. Define two $\nu$-measurable operator-valued
functions $D(\cdot)$ and $\tilde{A}(\cdot)$ on $\sigma(D^{}_{0})$ of
the form
$$D(\gamma)=\begin{cases}
\mathrm{diag}(\gamma,\gamma,\gamma)\in M^{}_{3}(\mathbb{C}),&\forall\gamma\in\Gamma^{}_{1};\\
\mathrm{diag}(\gamma,\gamma,0)\in M^{}_{3}(\mathbb{C}),&\forall\gamma\in\Gamma^{}_{2};\\
\mathrm{diag}(0,0,\gamma)\in M^{}_{3}(\mathbb{C}),&\forall\gamma\in\Gamma^{}_{3};\\
\end{cases} \eqno{(4.9)}$$ and $\tilde{A}(\gamma)$
to be the restriction of $A\circ D^{-1}_{i}(\gamma)$ on
$\mathrm{ran}(D(\gamma))=\mathscr{H}(\gamma)$ for
$\gamma\in\Gamma^{}_{i}$, $i=1,2,3$. Since $D^{}_{i}$ is injective
for $i=1,2,3$, the function $D(\cdot)$ is well-defined. We can
verify that there exists a unitary operator $U$ such that
$UD^{}_{0}U^{*}_{}=D|^{}_{\mathscr{H}}$, $UAU^{*}_{}=\tilde{A}$ and
$\tilde{A}(D|^{}_{\mathscr{H}})=D|^{}_{\mathscr{H}}\tilde{A}$. By
(\cite{Shi_2}, Lemma $3.1$), we obtain that $\tilde{A}(\gamma)$ is
strongly irreducible a.e.~$[\nu]$ on $\sigma(D^{}_{0})$. By a
similar argument, we can prove the `if only' part for general cases.

On the other hand, if $A\in M^{}_{n^{}_{}}(L^{\infty}(\mu))$ and
there exists a unitary operator $U$ such that $UAU^{*}$ is a direct
integral of strongly irreducible operators with respect to a
diagonal algebra $\mathscr{D}^{}_{}$ and
$U^{*}_{}\mathscr{D}^{}_{}U\subseteq
M^{}_{n^{}_{}}(L^{\infty}(\mu))$, then it is sufficient to show that
the set of projections $\mathscr{P}$ in $U^{*}_{}\mathscr{D}^{}_{}U$
is a bounded maximal abelian set of idempotents in the relative
commutant $\{A\}^{\prime}_{}\cap M^{}_{n^{}_{}}(L^{\infty}(\mu))$.
If there exists a nontrivial idempotent $P$ in
$\{A\}^{\prime}_{}\cap
M^{}_{n^{}_{}}(L^{\infty}(\mu))\cap\mathscr{P}^{\prime}_{}$ and $P$
not in $\mathscr{P}$, then $UPU^{*}_{}(\gamma)$ is not trivial for
almost every $\gamma$ on $\Gamma$ with respect to $\nu$, where
$\Gamma$ is the index set for $\mathscr{D}$ in the direct integral
form and $\nu$ is a finite positive regular Borel measure applied in
the direct integral form. Let $\Gamma^{}_{P}$ be a subset of
$\Gamma$ with $\nu(\Gamma^{}_{P})>0$ such that $UPU^{*}_{}(\gamma)$
is nontrivial for every $\gamma\in\Gamma^{}_{P}$. Since the equality
$UAU^{*}_{}(\gamma)UPU^{*}_{}(\gamma)=UPU^{*}_{}(\gamma)UAU^{*}_{}(\gamma)$
holds for almost every $\gamma$ in $\Gamma$ with respect to $\nu$.
Then $UAU^{*}_{}(\gamma)$ is not strongly irreducible for almost
every $\gamma$ in $\Gamma^{}_{P}$ with respect to $\nu$. This is a
contradiction. Therefore $\mathscr{P}$ is a bounded maximal abelian
set of idempotents in the relative commutant $\{A\}^{\prime}_{}\cap
M^{}_{n^{}_{}}(L^{\infty}(\mu))$. By Proposition $2.5$ and the
preceding argument, we prove the `if' part.
\end{proof}

Based on Proposition $2.8$, we investigate the local $K$-theory for
the relative commutant $\{A\}^{\prime}_{}\cap
M^{}_{n^{}_{}}(L^{\infty}(\mu))$, where $A$ in
$M^{}_{n^{}_{}}(L^{\infty}(\mu))$ and $\{A\}^{\prime}_{}\cap
M^{}_{n^{}_{}}(L^{\infty}(\mu))$ contains a finite frame. The
following lemma shows the $K^{}_{0}$ group of the relative commutant
of a building block in $(2.30)$.

\begin{lemma}
Let $A$ be an upper triangular operator-valued matrix in
$M^{}_{n^{}_{}}(L^{\infty}(\mu))$ of the form
$$A^{}_{{}}=\begin{pmatrix}
f^{{}}_{}&f^{{}}_{12}&\cdots&f^{{}}_{1n^{}_{}}\\
0&f^{{}}_{}&\cdots&f^{{}}_{2n^{}_{}}\\
\vdots&\vdots&\ddots&\vdots\\
0&0&\cdots&f^{{}}_{}\\
\end{pmatrix}^{}_{n^{}_{}\times n^{}_{}}
\begin{matrix}
L^{2}(\mu)\\
L^{2}(\mu)\\
\vdots\\
L^{2}(\mu)\\
\end{matrix} \eqno{(4.10)}$$
such that $f^{{}}_{i,i+1}$ is supported on $\Lambda$ for
$i=1,2,\ldots,n-1$. Then the $K^{}_{0}$ group of the relative
commutant $\{A\}^{\prime}_{}\cap M^{}_{n^{}_{}}(L^{\infty}(\mu))$ is
of the form
$$K^{}_{0}(\{A\}^{\prime}_{}\cap
M^{}_{n^{}_{}}(L^{\infty}(\mu)))\cong\{\phi:\Lambda\rightarrow\mathbb{Z}|\phi\mbox{
is bounded Borel}\}. \eqno{(4.11)}$$
\end{lemma}

\begin{proof}
By Lemma $3.2$, we obtain that for every $B$ in
$\{A\}^{\prime}_{}\cap M^{}_{n^{}_{}}(L^{\infty}(\mu))$, $B$ is of
the upper triangular form
$$B^{}_{{}}=\begin{pmatrix}
\phi^{{}}_{}&\phi^{{}}_{12}&\cdots&\phi^{{}}_{1n^{}_{}}\\
0&\phi^{{}}_{}&\cdots&\phi^{{}}_{2n^{}_{}}\\
\vdots&\vdots&\ddots&\vdots\\
0&0&\cdots&\phi^{{}}_{}\\
\end{pmatrix}^{}_{n^{}_{}\times n^{}_{}}
\begin{matrix}
L^{2}(\mu)\\
L^{2}(\mu)\\
\vdots\\
L^{2}(\mu)\\
\end{matrix}. \eqno{(4.12)}$$
Therefore, every $B$ in $\{A^{(m)}_{}\}^{\prime}_{}\cap
M^{}_{mn^{}_{}}(L^{\infty}(\mu))$ is of the form
$$B=\begin{pmatrix}
B^{}_{11}&B^{}_{12}&\cdots&B^{}_{1m}\\
B^{}_{21}&B^{}_{22}&\cdots&B^{}_{2m}\\
\vdots&\vdots&\ddots&\vdots\\
B^{}_{m1}&B^{}_{m2}&\cdots&B^{}_{mm}\\
\end{pmatrix}, \eqno{(4.13)}$$ where $B^{}_{kl}\in
M^{}_{n^{}_{}}(L^{\infty}(\mu))$ and
$A^{}_{}B^{}_{kl}=B^{}_{kl}A^{}_{}$ for every $k,l$ in
$\{1,\ldots,m\}$. Thus $B^{}_{kl}$ is of the upper triangular form
as in $(4.12)$
$$B^{}_{kl}=\begin{pmatrix}
B^{kl}_{11}&B^{kl}_{12}&\cdots&B^{kl}_{1n}\\
0&B^{kl}_{22}&\cdots&B^{kl}_{2n}\\
\vdots&\vdots&\ddots&\vdots\\
0&0&\cdots&B^{kl}_{nn}\\
\end{pmatrix}^{}_{n\times n}
\begin{matrix}
L^{2}(\mu)\\
L^{2}(\mu)\\
\vdots\\
L^{2}(\mu)\\
\end{matrix}, \eqno{(4.14)}$$ where we treat $B^{kl}_{ij}$ as an
element in $L^{\infty}(\mu)$ and $B^{kl}_{ii}=B^{kl}_{jj}$ for
$k,l=1,\ldots,m$ and $i,j=1,\ldots,n$.

There exists a unitary operator $U$ in
$M^{}_{mn^{}_{}}(L^{\infty}(\mu))$ which is a composition of
finitely many row-switching transformations such that $UBU^{*}(=C)$
as an element in $\{U(A^{(m)}_{})U^{*}\}^{\prime}\cap
M^{}_{mn^{}_{}}(L^{\infty}(\mu))$ is of the upper triangular form
$$C=\begin{pmatrix}
C^{}_{11}&C^{}_{12}&\cdots&C^{}_{1n}\\
0&C^{}_{22}&\cdots&C^{}_{2n}\\
\vdots&\vdots&\ddots&\vdots\\
0&0&\cdots&C^{}_{nn}\\
\end{pmatrix}, \eqno{(4.15)}$$ where for $i,j=1,\ldots,n$, every
$C^{}_{ij}$ is of the form
$$C^{}_{ij}=\begin{pmatrix}
B^{11}_{ij}&B^{12}_{ij}&\cdots&B^{1m}_{ij}\\
B^{21}_{ij}&B^{22}_{ij}&\cdots&B^{2m}_{ij}\\
\vdots&\vdots&\ddots&\vdots\\
B^{m1}_{ij}&B^{m2}_{ij}&\cdots&B^{mm}_{ij}\\
\end{pmatrix}^{}_{m\times m}
\begin{matrix}
L^{2}(\mu)\\
L^{2}(\mu)\\
\vdots\\
L^{2}(\mu)\\
\end{matrix}. \eqno{(4.16)}$$ The equality $C^{}_{ii}=C^{}_{jj}$
holds for $i,j=1,\ldots,n$ and $C^{}_{11}$ can be every element in
$M^{}_{m^{}_{}}(L^{\infty}(\mu))$. This fact is different from the
counterpart mentioned from $(3.30)$ to $(3.32)$.

If $B$ is an idempotent, then so is $C$. By a similar proof from
$(3.30)$ to $(3.39)$, we obtain an invertible element $X$ in
$\{U(A^{(m)}_{})U^{*}\}^{\prime}\cap
M^{}_{mn^{}_{}}(L^{\infty}(\mu))$ such that $XCX^{-1}_{}$ is of the
diagonal form
$$XCX^{-1}_{}=\begin{pmatrix}
C^{}_{11}&0&\cdots&0\\
0&C^{}_{22}&\cdots&0\\
\vdots&\vdots&\ddots&\vdots\\
0&0&\cdots&C^{}_{nn}\\
\end{pmatrix} \eqno{(4.17)}$$ where $C^{}_{ii}$s are the same
idempotent for $i=1,\ldots,n$. Furthermore, we assume that
$$\mathrm{Tr}(C^{}_{ii})(\lambda)=r \eqno{(4.18)}$$ for an integer
$1\leq r\leq m$ a.e.~$[\mu]$ on $\Lambda$. Then by Lemma $2.3$,
there exists an invertible element $Y$ in
$M^{}_{m^{}_{}}(L^{\infty}(\mu))$ such that
$$YC^{}_{ii}Y^{-1}_{}=I^{}_{M^{}_{r^{}_{}}(L^{\infty}(\mu))}\oplus 0.
\eqno{(4.19)}$$ Note that $Y^{(n)}_{}$ is an invertible element in
$\{U(A^{(m)}_{})U^{*}\}^{\prime}\cap
M^{}_{mn^{}_{}}(L^{\infty}(\mu))$. In this sense, we form an abelian
semigroup by equivalence classes of idempotents in
$\cup^{\infty}_{m=1}\{A^{(m)}_{}\}^{\prime}\cap
M^{}_{mn^{}_{}}(L^{\infty}(\mu))$ up to similarity. Thus this
semigroup is isomorphic to the abelian semigroup
$$\{\phi:\Lambda\rightarrow\mathbb{N}|\phi\mbox{
is bounded Borel}\} \eqno{(4.20)}$$ by an isomorphism induced by the
function $\mathrm{Tr}(\cdot)$ defined in $(2.25)$. Then by a routine
computation, we obtain the $K^{}_{0}$ group of
$\{A^{}_{}\}^{\prime}\cap M^{}_{n^{}_{}}(L^{\infty}(\mu))$ is of the
form formulated as in $(4.11)$.
\end{proof}

For a general case that $A$ in $M^{}_{n^{}_{}}(L^{\infty}(\mu))$ and
$\{A\}^{\prime}_{}\cap M^{}_{n^{}_{}}(L^{\infty}(\mu))$ contains a
finite frame, the $K^{}_{0}$ group of $\{A\}^{\prime}_{}\cap
M^{}_{n^{}_{}}(L^{\infty}(\mu))$ may not be of a neat form as in
$(4.11)$. We see this from the following example.

\begin{example}
Let $A$ in $M^{}_{4}(L^{\infty}(\mu))$ be of the form
$$A=\begin{pmatrix}
A^{}_{1}&{0}\\
{0}&A^{}_{2}\\
\end{pmatrix},\mbox{ and }
A^{}_{1}=\begin{pmatrix}
f&{f^{}_{1}}\\
{0}&f\\
\end{pmatrix}
\begin{matrix}
L^{2}(\mu)\\
L^{2}(\mu)\\
\end{matrix},\,
A^{}_{2}=\begin{pmatrix}
f&{f^{}_{2}}\\
{0}&f\\
\end{pmatrix}
\begin{matrix}
L^{2}(\mu)\\
L^{2}(\mu)\\
\end{matrix}, \eqno{(4.21)}$$ where $f^{}_{1}(\lambda)=\lambda$ and
$f^{}_{2}(\lambda)=1$ for every $\lambda$ in $\Lambda$ and we assume
that $\Lambda=[0,1]$. Note that $0$ is in $\sigma(f^{}_{1})$. For
every $\epsilon>0$, the relation
$f^{}_{1}(\{\lambda:|\lambda|<\epsilon\})\subseteq
\{\lambda:|f^{}_{1}(\lambda)|<\epsilon\}$ holds for the open ball
$\{\lambda:|\lambda|<\epsilon\}$ with
$\mu\{\lambda:|\lambda|<\epsilon\}>0$. Then the point $0$ in the
unit interval $[0,1]$ makes the $K^{}_{0}$ group of
$\{A\}^{\prime}_{}\cap M^{}_{4^{}_{}}(L^{\infty}(\mu))$ different
from the form as in $(4.11)$. We define an equivalent relation
`$\sim$' in the group
$$\mathscr{G}=\{\phi:\Lambda\rightarrow\mathbb{Z}^{(2)}_{}|\phi\mbox{
is bounded Borel}\}. \eqno{(4.22)}$$ That is for
$\phi=(\phi^{}_{1},\phi^{}_{2})$ and
$\psi=(\psi^{}_{1},\psi^{}_{2})$ in $\mathscr{G}$, $\phi$ is said to
be equivalent to $\psi$ (denoted by $\phi\sim\psi$) if there exists
an $\epsilon>0$ such that:
\begin{enumerate}
\item $\phi(\lambda)=\psi(\lambda)$ for $|\lambda|<\epsilon$;
\item
$\phi^{}_{1}(\lambda)+\phi^{}_{2}(\lambda)
=\psi^{}_{1}(\lambda)+\psi^{}_{2}(\lambda)$ for
$|\lambda|\geq\epsilon$.
\end{enumerate}
Thus by a routine computation, we obtain that
$$K^{}_{0}(\{A\}^{\prime}_{}\cap M^{}_{4^{}_{}}(L^{\infty}(\mu)))
\cong\mathscr{G}\slash\sim.  \eqno{(4.23)}$$ For a general case, we
assume that the support of $f^{}_{i}$ is $\Lambda$ and
$0\in\sigma(f^{}_{i})$ for $i=1,2$. Let
$\{p^{i}_{k}\}^{}_{i=1,2;k\in\mathbb{N}}$ be the set of points in
$\Lambda$ such that:
\begin{enumerate}
\item $f^{}_{i}(p^{i}_{k})=0$ for $i=1,2$ and every $k$ in
$\mathbb{N}$;
\item For every $p^{i}_{k}$ and $\epsilon>0$, there exists a
$\delta>0$ such that $|f^{}_{i}(p)|<\epsilon$ for
$|p-p^{i}_{k}|<\delta$.
\end{enumerate}
Essentially, the equivalent relation mentioned preceding $(4.23)$ is
determined by the set of $\{p^{i}_{k}\}^{}_{k\in\mathbb{N}}$ in
$\Lambda$ chosen with respect to the $1$-diagonal entries of
$A^{}_{1}$ and $A^{}_{2}$.
\end{example}

For upper triangular forms of different sizes, we obtain the
following lemma.

\begin{lemma}
Let $A=A^{}_{1}\oplus A^{}_{2}$ and $A^{}_{i}$ in
$M^{}_{n^{}_{i}}(L^{\infty}(\mu))$ is of the form
$$A^{}_{i}=\begin{pmatrix}
A^{i}_{11}&A^{i}_{12}&\cdots&A^{i}_{1n^{}_{i}}\\
0&A^{i}_{22}&\cdots&A^{i}_{2n^{}_{i}}\\
\vdots&\vdots&\ddots&\vdots\\
0&0&\cdots&A^{i}_{n^{}_{i}n^{}_{i}}\\
\end{pmatrix}^{}_{n^{}_{i}\times n^{}_{i}}
\begin{matrix}
L^{2}(\mu)\\
L^{2}(\mu)\\
\vdots\\
L^{2}(\mu)\\
\end{matrix}, \eqno{(4.24)}$$ such that
\begin{enumerate}
\item $n^{}_{1}>n^{}_{2}$ and $n=n^{}_{1}+n^{}_{2}$;
\item $A^{i}_{kk}=A^{i}_{ll}$ for $i=1,2$ and $k,l=1,\ldots,n^{}_{i}$;
\item the support of $A^{i}_{k,k+1}$ is $\Lambda$ for $i=1,2$ and
$k=1,\ldots,n^{}_{i}-1$.
\end{enumerate}
Then the $K^{}_{0}$ group of the relative commutant
$\{A\}^{\prime}_{}\cap M^{}_{n^{}_{}}(L^{\infty}(\mu))$ is of the
form
$$K^{}_{0}(\{A\}^{\prime}_{}\cap
M^{}_{n^{}_{}}(L^{\infty}(\mu)))\cong
\{\phi:\Lambda\rightarrow\mathbb{Z}^{(2)}_{}|\phi\mbox{ is bounded
Borel}\}. \eqno{(4.25)}$$
\end{lemma}

The proof of this lemma is a routine computation by applying Lemma
$3.2$, Lemma $3.3$, Lemma $3.4$, Lemma $3.6$ and Lemma $3.8$. By
Proposition $2.8$, to tell when two operators in
$M^{}_{n^{}_{}}(L^{\infty}(\mu))$ are similar to each other in
$M^{}_{n^{}_{}}(L^{\infty}(\mu))$, it is sufficient to tell when two
building blocks as in $(2.31)$ are similar to each other. For two
building blocks in $M^{}_{n^{}_{}}(L^{\infty}(\mu))$, we present a
method to distinguish them by the $K$-theory of the relative
commutant of the direct sum of this two building blocks.

\begin{proposition}
Let $A=A^{}_{1}\oplus A^{}_{2}$ and $A^{}_{i}$ in
$M^{}_{n^{}_{}}(L^{\infty}(\mu))$ is of the form
$$A^{}_{i}=\begin{pmatrix}
A^{i}_{11}&A^{i}_{12}&\cdots&A^{i}_{1n^{}_{}}\\
0&A^{i}_{22}&\cdots&A^{i}_{2n^{}_{}}\\
\vdots&\vdots&\ddots&\vdots\\
0&0&\cdots&A^{i}_{n^{}_{}n^{}_{}}\\
\end{pmatrix}^{}_{n^{}_{}\times n^{}_{}}
\begin{matrix}
L^{2}(\mu)\\
L^{2}(\mu)\\
\vdots\\
L^{2}(\mu)\\
\end{matrix}, \eqno{(4.26)}$$ such that
\begin{enumerate}
\item $A^{i}_{kk}=A^{i}_{ll}$ for $i=1,2$ and $k,l=1,\ldots,n^{}_{}$;
\item the support of $A^{i}_{k,k+1}$ is $\Lambda$ for $i=1,2$ and
$k=1,\ldots,n^{}_{}-1$.
\end{enumerate}
If the $K^{}_{0}$ group of the relative commutant
$\{A\}^{\prime}_{}\cap M^{}_{2n^{}_{}}(L^{\infty}(\mu))$ is of the
form
$$K^{}_{0}(\{A\}^{\prime}_{}\cap
M^{}_{2n^{}_{}}(L^{\infty}(\mu)))\cong
\{\phi:\Lambda\rightarrow\mathbb{Z}^{}_{}|\phi\mbox{ is bounded
Borel}\}, \eqno{(4.27)}$$ then $A^{}_{1}$ and $A^{}_{2}$ are similar
to each other in $M^{}_{n^{}_{}}(L^{\infty}(\mu))$.
\end{proposition}

By the methods we applied in Example $4.5$ and other lemmas in this
section, we can prove this proposition.

\bibliographystyle{amsplain}

\end{document}